\documentclass[11pt,reqno]{amsart}
\usepackage{amssymb}
\usepackage{amsmath}
\usepackage{amscd}
\usepackage{amsthm}
\usepackage{esint}
\usepackage{bm}
\usepackage{url}
\usepackage{xcolor}

\usepackage[linktocpage=true,colorlinks,citecolor=magenta,linkcolor=blue,urlcolor=magenta]{hyperref}
\usepackage[margin=1.1in]{geometry}
\newtheorem{thm}{Theorem}[section]
\newtheorem{corr}[thm]{Corollary}
\newtheorem{lem}[thm]{Lemma}
\newtheorem{prop}[thm]{Proposition}

\theoremstyle{definition}

\newtheorem*{ack}{Acknowledgment}

\numberwithin{equation}{section}
\allowdisplaybreaks
\def\R{\mathbb R}

\def\SS{\mathbb S}

\def\pt{\partial}

\begin{document}
\title[Fully nonlinear anisotropic curvature flow]{A fully nonlinear locally constrained anisotropic curvature flow}
\author[Y. Wei]{Yong Wei}
\address{School of Mathematical Sciences, University of Science and Technology of China, Hefei 230026, P.R. China}
\email{\href{mailto:yongwei@ustc.edu.cn}{yongwei@ustc.edu.cn}}

\author[C. Xiong]{Changwei~Xiong}
\address{College of Mathematics, Sichuan University, Chengdu 610065,  P.R. China}
\email{\href{mailto:changwei.xiong@scu.edu.cn}{changwei.xiong@scu.edu.cn}}



\maketitle

\begin{abstract}
Given a smooth positive function $F\in C^{\infty}(\mathbb{S}^n)$ such that the square of its positive $1$-homogeneous extension on $\mathbb{R}^{n+1}\setminus \{0\}$ is uniformly convex, the Wulff shape $W_F$ is a smooth uniformly convex body in the Euclidean space $\mathbb{R}^{n+1}$ with $F$ being the support function of the boundary $\partial W_F$. In this paper, we introduce the fully nonlinear locally constrained anisotropic curvature flow
\begin{equation*}
\frac{\partial }{\partial t}X=(1-E_k^{1/k}\sigma_F)\nu_F,\quad k=2,\cdots,n
\end{equation*}
in the Euclidean space, where $E_k$ denotes the normalized  $k$th anisotropic mean curvature with respect to the Wulff shape $W_F$, $\sigma_F$ the anisotropic support function and $\nu_F$ the outward anisotropic unit normal of the evolving hypersurface. We show that starting from a smooth, closed and strictly convex hypersurface in $\mathbb{R}^{n+1}$ ($n\geq 2$), the smooth solution of the flow exists for all positive time and converges smoothly and exponentially to a scaled Wulff shape.   A nice feature of this flow is that it improves a certain isoperimetric ratio. Therefore by the smooth convergence of the above flow, we provide a new proof of a class of the Alexandrov--Fenchel inequalities for anisotropic mixed volumes of smooth convex domains in the Euclidean space.
\end{abstract}

\tableofcontents

\section{Introduction}


Given a smooth positive function $F\in C^{\infty}(\mathbb{S}^n)$ such that the square of its positive $1$-homogeneous extension on $\mathbb{R}^{n+1}\setminus \{0\}$ is uniformly convex, the Wulff shape $W_F$ is defined as the set
\begin{equation*}
  W_F=\bigcap_{\nu\in \mathbb{S}^n}\left\{x\in \mathbb{R}^{n+1}: x\cdot\nu\leq F(\nu)\right\},
\end{equation*}
which uniquely solves the isoperimetric problem for the interface energy functional (or called the anisotropic perimeter)
\begin{equation*}
  \mathcal{E}(K)=\int_{\partial^*K}F(\nu)d\mathcal{H}^n
\end{equation*}
among sets of finite perimeter and of given volume. Here $\partial^*K$ denotes the reduced boundary of $K$ and $\nu$ denotes the outward unit normal vector field.  See \cite{BM94,FM91,Tay78} and \cite{FMP2010}. The function $F$ on $\SS^n$ is called the support function of the Wulff shape $W_F$. The assumption on $F$ ensures that $W_F$ is a smooth uniformly convex body in $\mathbb{R}^{n+1}$, and its boundary $\Sigma_F=\partial W_F$ is also called the Wulff shape determined by the support function $F$. When $F\equiv 1$ on $\mathbb{S}^n$, the Wulff shape is just the unit round sphere in $\mathbb{R}^{n+1}$.

The Wulff shape plays the similar role in the anisotropic geometry (also called the relative or Minkowski differential geometry) as the one the unit sphere plays in the Euclidean geometry; see \cite{BF87} and references therein.  In fact, the Wulff shape allows us to define the anisotropic normal of a smooth hypersurface $M^n$ in $\mathbb{R}^{n+1}$ by a map $\nu_{F}:\ M\to \Sigma_F$, which takes each point $x\in M$ to the point in $\Sigma_F$ with the same oriented tangent plane.  The anisotropic Weingarten map ${\mathcal W}_{F}$ is the derivative of the anisotropic normal $\nu_{F}$, which is a linear map from $T_xM$ to itself at each point.  The eigenvalues $\kappa=(\kappa_1,\cdots,\kappa_n)$ of ${\mathcal W}_{F}$ are called the anisotropic principal curvatures. We define the $k$th anisotropic mean curvature of $M$ as the normalized $k$th elementary symmetric function $E_k$ of the anisotropic principal curvatures $\kappa$:
 \begin{equation*}
   {E}_{k}=E_k(\kappa)={\binom{n}{k}}^{-1}\sum_{1\leq i_1<\cdots<i_k\leq n}\kappa_{i_1}\cdots \kappa_{i_k},\quad k=1,\cdots,n.
 \end{equation*}
In particular, $H_F=nE_1(\kappa)$ denotes the anisotropic mean curvature of $M$. On the other hand, we can define the anisotropic support function $\sigma_F:M\rightarrow \R $ by the decomposition
\begin{equation*}
X=\sigma_F \nu_F+V^i\pt_iX, \quad V^i\pt_iX\in T_xM
\end{equation*}
of the position vector $X$. If we denote by $\nu$ the outward unit normal of $M$ in the Euclidean geometry, the anisotropic support function $\sigma_F$ is related to the isotropic support function $\sigma=\langle X,\nu\rangle$ by the equation $\sigma_F=\sigma/{F}(\nu)$.

In \cite{WeiX20}, we introduced a volume-preserving anisotropic mean curvature flow $X:M\times [0,T)\to \mathbb{R}^{n+1}, n\geq 2$ of star-shaped hypersurfaces in $\mathbb{R}^{n+1}$ which satisfies
\begin{equation}\label{flow0}
\begin{cases}
\dfrac{\pt }{\pt t}X=(n-H_F\sigma_F)\nu_F,\\
X(\cdot,0)= X_0(\cdot).
\end{cases}
\end{equation}
This flow was inspired by the isotropic one introduced earlier by Guan and Li \cite{GL15}. Denote by $M_t=X(M,t)$ the evolving hypersurface and by $\Omega_t$ the domain enclosed by $M_t$. By the anisotropic Minkowski identity \eqref{eq2.Minkowski}, the volume of $\Omega_t$ along the flow \eqref{flow0} remains a constant
\begin{equation*}
  \frac d{dt}|\Omega_t|~=~\int_{M_t}(n-H_F\sigma_F)d\mu_F=0,
\end{equation*}
where $d\mu_F=F(\nu)d\mu$ is the anisotropic area form on $M$ determined by the anisotropy $F$. We proved that the flow \eqref{flow0} starting from any star-shaped and closed hypersurface exists for all time $t\in[0,\infty)$ and converges smoothly to a scaled Wulff shape. As an application, we provided a new proof of a class of the Alexandrov--Fenchel inequalities for smooth convex domains in the Euclidean space.

The flow \eqref{flow0} is quasi-linear as the speed involves only the anisotropic mean curvature and the anisotropic support function. In this paper, we introduce a fully nonlinear locally constrained anisotropic curvature flow:
\begin{equation}\label{flow-1}
\frac{\pt }{\pt t}X=(1-E_k^{1/k}(\kappa)\sigma_F)\nu_F,\quad k=2,\cdots,n.
\end{equation}
When $k=1$, the flow \eqref{flow-1} reduces to the flow \eqref{flow0} up to a constant $n$ which is volume-preserving. However, for $k\geq 2$ the flow \eqref{flow-1} does not preserve the enclosed volume or any mixed volume generally. Recall that the mixed volumes of a smooth convex domain $\Omega$ relative to the Wulff shape $W_F$ can be expressed as (see \S \ref{sec:2-AF})
\begin{equation}\label{s1.Vk}
  V_{n-k}(\Omega,W_F)=\int_{M} E_k(\kappa)d\mu_F, \quad k=0,\cdots, n-1,
\end{equation}
and $V_{n+1}(\Omega,W_F)=(n+1)|\Omega|$, $V_0(\Omega,W_F)=(n+1)|W_F|$.  Define the following higher order anisotropic isoperimetric ratio
\begin{equation}\label{s1.Ik}
  \mathcal{I}_k(\Omega,W_F)=\frac{V_{n+2-k}(\Omega, W_F)}{V_{n+1}(\Omega,W_F)^{\frac{n+2-k}{n+1}}},\quad k=2,\cdots,n.
\end{equation}
We show that $\mathcal{I}_k(\Omega_t,W_F)$  is monotone non-increasing along the flow \eqref{flow-1}.   This is a motivation for us to introduce the flow \eqref{flow-1}.

Our main result is as follows.

\begin{thm}\label{s1-thm1}
Let $\Sigma_F \subset \mathbb{R}^{n+1}, n\geq 2,$ be a given smooth, closed and strictly convex hypersurface enclosing the origin with the support function $F\in C^\infty(\mathbb{S}^n)$. Let $M_0=\partial\Omega_0\subset \mathbb{R}^{n+1}$ be a smooth, closed and strictly convex hypersurface with the origin contained in $\Omega_0$. Then the flow \eqref{flow-1} starting from $M_0$ has a unique smooth solution for all time $t\in [0,\infty)$ and the evolving hypersurface $M_t$ converges smoothly and exponentially to $\bar{r}\Sigma_F$ for some $\bar{r}>0$ as $t\to\infty$.
\end{thm}

Since the initial hypersurface $M_0$ is strictly convex and encloses the origin, we know that $M_0$ is star-shaped with respect to the origin. This is equivalent to that the support function $\sigma=\langle X,\nu\rangle $ is positive everywhere on $M_0$, where $X$ and $\nu$ denote the position vector and the outward unit normal of $M_0$ respectively. It follows that the anisotropic support function $\sigma_F=\sigma/{F(\nu)}$ is positive on $M_0$ as well. The proof of Theorem \ref{s1-thm1} consists of the following steps. First, we apply the maximum principle to the evolution equation of $\sigma_F$ to show that $\sigma_F$ is bounded from below by a positive constant along the flow \eqref{flow-1}, which means that the evolving hypersurface $M_t$ remains to be star-shaped with respect to the origin for positive time $t>0$. Then we can write $M_t$ as a graph of the radial function $\rho$ over the sphere $\mathbb{S}^n$, and the flow \eqref{flow-1} is equivalent to a fully nonlinear parabolic equation of the function $\gamma=\log \rho $ on $\mathbb{S}^n$. The short time existence of the flow \eqref{flow-1} follows immediately. To derive the a priori estimates of the solution $M_t$, we first apply the comparison principle to the equation \eqref{flow-1} to obtain the $C^0$ estimate.  Then the uniform positive lower bound of the anisotropic support function $\sigma_F$ together with the $C^0$ estimate implies the a priori $C^1$ estimate for the solution.

Since the flow \eqref{flow-1} is fully nonlinear, we still need the $C^2$ estimate in order to get the higher order regularity estimate. For this purpose, we employ the anisotropic Gauss map parametrization which we will review in \S \ref{sec:Gauss}. As we already have the $C^1$ estimate, the $C^2$ estimate is equivalent to the bounds on the anisotropic principal curvatures. Equivalently, it suffices to estimate the upper bound on the anisotropic principal radii of curvature  $\tau_1,\cdots,\tau_n$, which are the eigenvalues of the following $(0,2)$-tensor on the Wulff shape $\Sigma_F$,
\begin{equation}\label{s1:tau-def}
  \tau_{ij}[s]=\bar{\nabla}_i\bar{\nabla}_js+\bar{g}_{ij}s-\frac 12Q_{ijk}\bar{\nabla}_ks.
\end{equation}
Here $s$ is the anisotropic support function considered as a function on the Wulff shape $\Sigma_F$ which coincides with $\sigma_F$ introduced above. See Section \ref{sec:Gauss} for the detail. Section \ref{sec:C2} is devoted to proving the upper bound on $\tau_i$. The method is to apply the maximum principle to the evolution equation of $\tau_{ij}[s]$ together with an auxiliary function involving the anisotropic support function $s$ and its gradient. This is the most technical part in the proof, where an observation on smooth symmetric functions due to Guan, Shi and Sui \cite{Guanbo15} will be used.

Once the $C^2$ estimate is obtained, the higher order regularity estimate of the solution follows from the standard parabolic theory. This implies the long time existence of the smooth solution. The smooth convergence to a scaled Wulff shape follows from the improving of the isoperimetric ratio $ \mathcal{I}_k(\Omega,W_F)$. In \cite{GL15}, the exponential convergence of the flow \eqref{flow0} is based on the exponential decay of gradient $|\nabla^{\mathbb{S}}\gamma|$ of the radial function, which is an immediate consequence of the evolution of $|\nabla^{\mathbb{S}}\gamma|$ and the long time existence of the flow \eqref{flow0}. In the anisotropic case, the evolution of $|\nabla^{\mathbb{S}}\gamma|$ does not behave well. Instead, we will prove the exponential convergence by studying the linearization of the flow \eqref{flow-1} around the Wulff shape. The monotonicity of the volume $V_{n+1}(\Omega, W_F)$ and mixed volume $V_{n+2-k}(\Omega, W_F)$ will be used crucially to derive a faster decay of the average integral of anisotropic support function. With the help of a special class of Gagliardo-Nirenberg interpolation inequalities, we obtain the exponential convergence of the flow with the decay rate determined by the first non-zero eigenvalue of the self-adjoint operator \eqref{s9.L} on the Wulff shape.

As an application of Theorem \ref{s1-thm1}, we provide a new proof of the following Alexandrov--Fenchel inequalities.
\begin{corr}\label{s1-cor1}
Let $\Sigma_F=\partial W_F\subset \mathbb{R}^{n+1}, n\geq 2,$ be a smooth, closed and strictly convex hypersurface enclosing the origin with the support function $F\in C^\infty(\mathbb{S}^n)$. For any smooth, closed and strictly convex hypersurface $M$ which encloses a bounded domain $\Omega$, we have
\begin{equation}\label{s1:AF}
  \int_{M}E_k(\kappa)d\mu_F~\geq~(n+1)|\Omega|^{\frac {n-k}{n+1}}|W_F|^{\frac{k+1}{n+1}},\quad k=0,\dots,n-2,
\end{equation}
with the equality if and only if $M$ is a scaled Wulff shape. Here $|\Omega|$ and $|W_F|$ are the volume of  $\Omega$ and the Wulff shape $W_F$ respectively, $E_k(\kappa)$ denotes the $k$th normalized anisotropic mean curvature of $M$, and $d\mu_F$ the anisotropic area element.
\end{corr}

The paper is organized as follows. In \S \ref{sec2} we review some preliminaries on the anisotropic geometry, the mixed volumes and the Alexandrov--Fenchel inequalities. In \S \ref{sec:evl}, we derive the evolution equations along the flow \eqref{flow-1}. In \S \ref{sec:C01}, we derive the $C^0$ and $C^1$ estimates of the solution $M_t$. To derive the curvature estimate, we employ the anisotropic Gauss map parametrization which we describe in \S \ref{sec:Gauss}. Then in \S \ref{sec:C2}, we estimate the upper bound on the curvature function $E_k(\kappa)$ and the lower bound on the anisotropic principal curvatures. The examination of the monotonicity of $ \mathcal{I}_k(\Omega_t,W_F)$ will be given in \S \ref{sec:monot}. Finally, we complete the proofs of Theorem \ref{s1-thm1} and Corollary \ref{s1-cor1} in the \S \ref{sec:LTE} and \S \ref{sec:exp}. Throughout the paper the Einstein convention on the summation for indices is used unless otherwise stated.

\begin{ack}
The first author was supported by National Key Research and
Development Project SQ2020YFA070080 and a research grant from University of Science and
Technology of China. The second author was supported by Australian Laureate Fellowship FL150100126 of the Australian Research Council.
\end{ack}

\section{Preliminaries}\label{sec2}
In this section, we review some preliminaries of the anisotropic geometry, the mixed volumes and the Alexandrov--Fenchel inequalities.
\subsection{The Wulff shape}\label{sec:2-1}
Let $F$ be a smooth positive function on the standard sphere $\mathbb{S}^n$ such that the matrix
\begin{equation}\label{s2:A_F}
  A_{F}(x)~=~\nabla^{\mathbb{S}}\nabla^{\mathbb{S}} {F}(x)+F(x)g_{\mathbb{S}^n},\quad x\in \mathbb{S}^n
\end{equation}
is positive definite on $\SS^n$, where $\nabla^{\mathbb{S}}$ denotes the covariant derivative on $\mathbb{S}^n$. Then there exists a unique smooth strictly convex hypersurface $\Sigma_F$ given by
\begin{equation*}
  \Sigma_F=\{\phi(x)|\phi(x):=F(x)x+\nabla^{\mathbb{S}} {F}(x),~x\in \mathbb{S}^n\}
\end{equation*}
whose support function is given by $F$. Denote the closure of the enclosed domain of $\Sigma_F$ by $W_F$. We call $W_F$ (and $\Sigma_F$) the Wulff shape determined by the function $F\in C^{\infty}(\mathbb{S}^n)$. When $F$ is a constant, the Wulff shape is just a round sphere.

The smooth function $F$ on $\mathbb{S}^n$ can be extended homogeneously to a $1$-homogeneous function on $\mathbb{R}^{n+1}$ by
\begin{equation*}
  F(x)=|x|F({x}/{|x|}), \quad x\in \mathbb{R}^{n+1}\setminus\{0\}
\end{equation*}
and setting $F(0)=0$. Then it is easy to show that $\phi(x)=DF(x)$ for $x\in \mathbb{S}^n$, where $D$ denotes the gradient on $\mathbb{R}^{n+1}$. The homogeneous extension $F$ defines a Minkowski norm on $\mathbb{R}^{n+1}$, that is, $F$ is a norm on $\mathbb{R}^{n+1}$ and $D^2(F^2)$ is uniformly positive definite on $\mathbb{R}^{n+1}\setminus\{0\}$. We can define a dual Minkowski norm $F^0$ on $\mathbb{R}^{n+1}$ by
\begin{equation}\label{s2:gam0}
  F^0(z):=\sup_{x\neq 0}\frac{\langle x,z\rangle}{{F}(x)},\quad z\in \mathbb{R}^{n+1}.
\end{equation}
Then the Wulff shape $W_F$ can be written as
\begin{equation*}
W_F=\{z\in \R^{n+1}:{F}^0(z)\leq  1\},
\end{equation*}
and $\Sigma_F=\partial W_F=\{z\in \R^{n+1}: F^0(z)=1\}$.

\subsection{Anisotropic curvatures}\label{sec:2-2}
Let $M$ be a smooth hypersurface in the Euclidean space $\mathbb{R}^{n+1}$ with a  unit normal vector field $\nu$, and $\Sigma_F$ be the Wulff shape defined in \S \ref{sec:2-1}. We define the anisotropic unit normal of $M$ as the map $\nu_{F}: M\to \Sigma_F$ given by
\begin{equation*}
  \nu_{F}(x)~=~\phi(\nu(x))={F}(\nu(x))\nu(x)+\nabla^{\mathbb{S}} {F}(\nu(x))~\in ~\Sigma_F
\end{equation*}
for any point $x\in M$. It follows from the $1$-homogeneity of $F$ that $\nu_{F}(x)=DF(\nu(x))$. The anisotropic Weingarten map is a linear map
$$\mathcal{W}_{F}=d\nu_{F}: T_xM\to T_{\nu_{F}(x)}\Sigma_F.$$
Note that $\mathcal{W}_{F}=A_{F}\circ \mathcal{W}$, where $\mathcal{W}=d\nu=(h_i^j)$ is the Weingarten map of the hypersurface $M$, and that the eigenvalues of $\mathcal{W}$ are the (isotropic) principal curvatures $\lambda=(\lambda_1,\cdots,\lambda_n)$. The eigenvalues of $\mathcal{W}_{F}$ are called the anisotropic principal curvatures of $M$, and we denote them by $\kappa=(\kappa_1,\cdots,\kappa_n)$.  If we write $\mathcal{W}_{F}=(\hat{h}_i^j)$ in local coordinates, then
\begin{equation}\label{s2:hat-h}
  \hat{h}_i^j(x)=(A_{F}(\nu(x)))_i^kh_k^j(x).
\end{equation}
In particular, it follows from \eqref{s2:hat-h} that $E_n(\kappa)=\det(A_{F}(\nu))E_n(\lambda)$.

There is another formulation of the anisotropic curvatures, which was introduced by Ben~Andrews in \cite{And01} and reformulated by Chao~Xia in \cite{Xia13}. The idea is to define a new metric on $\mathbb{R}^{n+1}$ using the second derivatives of the dual function ${F}^0$ of the homogeneous extension of $F$. This new metric on the $\mathbb{R}^{n+1}$ is defined as
\begin{equation}\label{s2:G}
  G(z)(U,V):=\frac 12\sum_{i,j=1}^{n+1}\frac{\partial^2({F}^0)^2(z)}{\partial z^i\partial z^j}U^iV^j
\end{equation}
for $z\in \mathbb{R}^{n+1}\setminus\{0\}$ and $U,V\in T_{z}\mathbb{R}^{n+1}$.  The third derivatives of $F^0$ give a $(0,3)$-tensor
\begin{equation}\label{s2:Q-def}
  Q(z)(U,V,Y):=\frac 12\sum_{i,j,k=1}^{n+1}\frac{\partial^3({F}^0)^2(z)}{\partial z^i\partial z^j\partial z^k}U^iV^jY^k
\end{equation}
for $z\in \mathbb{R}^{n+1}\setminus\{0\}$ and $U,V,Y\in T_{z}\mathbb{R}^{n+1}$. The $1$-homogeneity of $F^0$ implies that
\begin{align}
  G(z)(z,z) =& 1,\quad G(z)(z,V)=0,\quad \mathrm{for}\quad z\in \Sigma_F, ~\mathrm{and} ~V\in T_{z}\Sigma_F, \nonumber\\
  Q(z)(z,U,V) =& 0,\quad \mathrm{for}\quad z\in \Sigma_F,~\mathrm{and}~U,V\in \mathbb{R}^{n+1}. \label{s2:Q-prop}
\end{align}

For a smooth hypersurface $M$ in $\mathbb{R}^{n+1}$, the anisotropic normal $\nu_{F}(x)$ lies in $\Sigma_F$. Then
\begin{align*}
  G(\nu_{F})(\nu_{F},\nu_{F}) =& 1,\quad G(\nu_{F})(\nu_{F},U)=0,\quad \mathrm{for} ~U\in T_xM, \\
   Q(\nu_{F})(\nu_{F},U,V)=&0,\quad \mathrm{for} ~U,V\in \mathbb{R}^{n+1}. 
\end{align*}
Thus $\nu_{F}(x)$ is perpendicular to $T_xM$ with respect to the metric $G(\nu_{F}(x))$. This induces a Riemannian metric $\hat{g}$ on $M$ from $(\mathbb{R}^{n+1},G)$ by
\begin{equation}\label{s2:g-hat}
\hat{g}(x):=G(\nu_{F}(x))|_{T_xM},\quad x\in M. 
\end{equation}
Then we can state the anisotropic Gauss and Weingarten formulas for $(M,\hat{g})\subset (\mathbb{R}^{n+1},G)$:
\begin{align}
  \partial_i\partial_jX=&-\hat{h}_{ij}\nu_F+\hat{\nabla}_{\partial_i}\partial_j+A_{ijk}\hat{g}^{k\ell}\partial_\ell X,\label{s2:Gauss-for}\\
  \partial_i\nu_F=&\hat{h}_{i}^k\partial_kX,\label{s2:AWeingart}
\end{align}
where $X$ is the position vector of $M$ in $\mathbb{R}^{n+1}$, $\hat{\nabla}$ is the Levi-Civita connection of $\hat{g}$, $\hat{h}_{ij}$ stands for the second fundamental form of $(M,\hat{g})\subset (\mathbb{R}^{n+1},G)$,
and the anisotropic Weingarten map relates to $\hat{h}_{ij}$ by $\mathcal{W}_{F}=(\hat{h}_i^j)=(\sum_k\hat{g}^{jk}\hat{h}_{ik})$.  We also have the anisotropic analogue of the Gauss equation and the Codazzi equation:
\begin{align}\label{s2:Gaus1}
  \hat{R}_{ijk\ell}~ =& ~\hat{h}_{ik}\hat{h}_{j\ell}-\hat{h}_{i\ell}\hat{h}_{jk}+\hat{\nabla}_{\ell}A_{jki}  -\hat{\nabla}_kA_{j\ell i}\nonumber\\
  &\quad +\hat{g}^{pm}A_{jkp}A_{m\ell i}-\hat{g}^{pm}A_{j\ell p}A_{mki},
\end{align}
\begin{equation}\label{s2:codazzi}
  \hat{\nabla}_k\hat{h}_{ij}+\hat{h}_j^\ell{A}_{\ell ki}~=~\hat{\nabla}_j\hat{h}_{ik}+\hat{h}_k^\ell{A}_{\ell ji},
\end{equation}
where $\hat{R}$ is the Riemannian curvature tensor of $\hat{g}$, and $A$ is a $(0,3)$-tensor
\begin{equation}\label{s2:A-def}
{A}_{ijk}=-\frac 12\left(\hat{h}_i^\ell Q_{jk\ell}+\hat{h}_j^\ell Q_{i\ell k}-\hat{h}_k^\ell Q_{ij\ell}\right),
\end{equation}
with $Q_{ijk}=Q(\nu_{F})(\partial_iX,\partial_jX,\partial_kX)$.  Note that by the definition \eqref{s2:Q-def}, $Q$ is totally symmetric in all three indices. Hence the tensor ${A}_{ijk}$ is symmetric for the first two indices.

When the hypersurface $M$ is the Wulff shape $\Sigma_F$, the anisotropic normal $\nu_{F}$ is just the position vector, and the anisotropic principal curvatures are all equal to $1$. Then $A_{ijk}=-Q_{ijk}/2$ which is symmetric in all indices.  We use the notations $\bar{g}$, $\bar{\nabla}$ and $\bar{R}$ to denote the induced metric on the Wulff shape $\Sigma_F$ from $(\mathbb{R}^{n+1},G)$, its Levi-Civita connection and curvature tensor respectively. Then $\bar{\nabla}_\ell Q_{ijk}$ is totally symmetric in four indices (Prop.~2.2 in \cite{Xia17-2}). In this case, the Gauss equation \eqref{s2:Gaus1} is simplified as
 \begin{equation}\label{s2:gauss-2}
  \bar{R}_{ijk\ell}=\bar{g}_{ik}\bar{g}_{j\ell}-\bar{g}_{i\ell}\bar{g}_{jk}+\frac 14\bar{g}^{pq}Q_{i\ell p}Q_{qjk}-\frac 14\bar{g}^{pq}Q_{ikp}Q_{j\ell q}.
\end{equation}

\subsection{Anisotropic support function}
Recall that the anisotropic support function $\sigma_F:{M}\rightarrow \R$ is defined by the following decomposition
\begin{equation*}
X=\sigma_F\nu_F+V^i\pt_iX,
\end{equation*}
where $V^i\pt_iX$ is the tangential part of the position vector $X$ along the hypersurface, and so
\begin{equation}\label{s2:spt-def}
  \sigma_F=G(\nu_F)(X,\nu_F).
\end{equation}
The anisotropic support function $\sigma_F$ is related to the Euclidean support function $\sigma=\langle X,\nu\rangle$ by
\begin{equation*}
  \sigma_F={\sigma}/{{F}(\nu)}.
\end{equation*}
For any smooth hypersurface $M$ in $\mathbb{R}^{n+1}$, we define the anisotropic area functional as
\begin{equation*}
  |M|_{F}:=\int_MF(\nu)d\mu,
\end{equation*}
where $d\mu$ is the area form of the induced metric on $M$ from the Euclidean space $\mathbb{R}^{n+1}$. We set $d\mu_{F}=F(\nu)d\mu$ and call it the anisotropic area form of $M$.

We recall the following generalized Minkowski identities.
\begin{prop}[see \cite{HL08}]
Let $M=\partial \Omega$ be a smooth closed hypersurface in $\mathbb{R}^{n+1}$. We have the following generalized Minkowski identities:
\begin{equation}
  \int_M\sigma_Fd\mu_F~=~(n+1)|\Omega|,
\end{equation}
and
\begin{equation}\label{eq2.Minkowski}
 \int_{M}E_{k+1}(\kappa)\sigma_Fd\mu_F~=~\int_ME_{k}(\kappa)d\mu_F
\end{equation}
for $k=0,1,\cdots,n-1$, where $E_k(\kappa)$ denotes the $k$th anisotropic mean curvature of $M$.
\end{prop}

\subsection{Mixed volumes and Alexandrov--Fenchel inequalities}\label{sec:2-AF}$\ $

Let $\Omega_1,\cdots,\Omega_{n+1}$ be convex bodies in $\mathbb{R}^{n+1}$. The Minkowski sum $\sum_{i=1}^{n+1}\epsilon_i\Omega_i$ is the set of points in $\mathbb{R}^{n+1}$ of the form $\sum_{i=1}^{n+1}\epsilon_ix_i$, where $x_i\in \Omega_i$ and $\epsilon_i$ are real numbers. The volume of this Minkowski sum is a polynomial of degree $n+1$ in $\epsilon_i$, and its coefficients are called \emph{mixed volumes}:
\begin{equation*}
 V[\Omega_1,\Omega_2,\cdots,\Omega_{n+1}]=\frac{1}{n!}\frac{\partial^{n+1}}{\partial\epsilon_1\cdots\partial\epsilon_{n+1}} \mathrm{Vol}(\sum_{i=1}^{n+1}\epsilon_i\Omega_i).
\end{equation*}
The mixed volumes satisfy a quadratic inequality called the Alexandrov--Fenchel inequality:
\begin{equation}\label{s2:AF}
  V[\Omega_1,\Omega_1,\Omega_3,\cdots,\Omega_{n+1}]V[\Omega_2,\Omega_2,\Omega_3,\cdots,\Omega_{n+1}]\leq V[\Omega_1,\Omega_2,\Omega_3,\cdots,\Omega_{n+1}]^2
\end{equation}
for any convex bodies $\Omega_1,\cdots,\Omega_{n+1}$ in $\mathbb{R}^{n+1}$.
See Schneider's book \cite[\S 7.3]{Sch14} for a detailed introduction.  The general inequality \eqref{s2:AF} has many important special cases, one of which is reviewed as follows. Given a fixed strictly convex body $W_F$, define
\begin{equation*}
  V_k(\Omega,W_F)=V[\underbrace{W_F,\cdots,W_F}_{n+1-k},\underbrace{\Omega,\cdots,\Omega}_k].
\end{equation*}
Then $V_{n+1}(\Omega,W_F)=(n+1)\mathrm{Vol}(\Omega)$ and $V_{0}(\Omega,W_F)=(n+1)\mathrm{Vol}(W_F)$. The inequality \eqref{s2:AF} implies the following inequalities:
\begin{equation}\label{s2:AF1}
  V_j^{k-i}(\Omega,W_F)~\geq~V_i^{k-j}(\Omega,W_F)V_k^{j-i}(\Omega,W_F)
\end{equation}
for any convex body $\Omega$ and any integers $i,j,k$ with $0\leq i<j<k\leq n+1$, which generalize the classical isoperimetric inequality. If $\partial\Omega$ is in class $C^2$, the mixed volumes $V_k(\Omega,W_F)$ are related to the following boundary anisotropic curvature integrals:
\begin{equation*}
 V_{n-k}(\Omega,W_F)=\int_{\partial\Omega}E_k(\kappa)d\mu_F,\quad k=0,\cdots,n,
\end{equation*}
where $\kappa=(\kappa_1,\cdots,\kappa_n)$ are the anisotropic principal curvatures of $\partial\Omega$ with respect to the Wulff shape $W_F$.  In the case that $W_F$ is the unit Euclidean ball $\mathbb{B}^{n+1}$, we denote the isotropic mixed volumes by $V_{k}(\Omega)=V_k(\Omega,\mathbb{B}^{n+1})$ for simplicity of the notation.  When $\partial\Omega$ is in class $C^2$,
\begin{equation*}
 V_{n-k}(\Omega)=\int_{\partial\Omega}E_k(\lambda)d\mu,\quad k=0,\cdots,n,
\end{equation*}
where $\lambda$ are the (isotropic) principal curvatures of $\partial\Omega$.

It is interesting to provide new proofs of the inequality \eqref{s2:AF1} (and the more general one \eqref{s2:AF}) by using geometric flows, and to prove them for a broader class of domains. Several progress has been made in this direction. For example, McCoy \cite{McC2005} used the mixed volume preserving flow to give a new proof of \eqref{s2:AF1} for the case $W_F=\mathbb{B}^{n+1}$, $i=0$ and $\partial\Omega$ is smooth and convex. Schulze \cite{Sch08} applied the flow by powers of the mean curvature to reprove the isoperimetric inequality for bounded domains in $\R^{n+1}$ ($n\leq 7$). Guan and Li \cite{GL} used the inverse curvature flow to prove \eqref{s2:AF1} for the case $W_F=\mathbb{B}^{n+1}$, $i=0$ and $\partial\Omega$ is smooth, $(n-j)$-convex and star-shaped. For the general Wulff shape $W_F$, Xia \cite{Xia17} proved the inequality \eqref{s2:AF1} for $i=0$, $j=n-1$, $k=n$ and a star-shaped, $F$-mean convex domain $\Omega$ using the inverse anisotropic mean curvature flow. The authors of the present paper reproved in \cite{WeiX20} the inequality \eqref{s2:AF1} for $i=0, k=n+1$ and any $j=2,\cdots,n$ for smooth and convex domains by applying the volume-preserving flow \eqref{flow0}.

\section{Evolution equations} \label{sec:evl}

We first review evolution equations for some geometric quantities on the hypersurface $M$ in $\mathbb{R}^{n+1}$. Let $X(\cdot,t):M \rightarrow \mathbb{R}^{n+1}$, $t\in [0,T)$, be a smooth family of embeddings satisfying
\begin{equation}\label{eq2.1}
\frac{\pt }{\pt t}X =\eta \nu_F,
\end{equation}
where $\eta$ is a time-dependent smooth function and  $\nu_F$ denotes the anisotropic normal of $M_t=X(M,t)$. The enclosed domain of $M_t$ is denoted by $\Omega_t$. The following equations were calculated in \cite{And01,Xia17}.
\begin{lem}[{\cite{And01,Xia17}}]
Under the equation \eqref{eq2.1}, we have
\begin{align}
\frac{\pt}{\pt t}d\mu_F&= \eta H_F d\mu_F,\label{s2:dmu}\\
\frac{\pt}{\pt t}\nu_F&= -\hat{\nabla} \eta,\label{s2:nu_F}\\
\frac{\pt}{\pt t} \hat{g}_{ij}&=2\eta \hat{h}_{ij}-Q_{ij}{}^k\hat{\nabla}_k \eta,\\
\frac{\pt}{\pt t}\hat{h}^{\ell}_k&= -\hat{\nabla}^\ell\hat{\nabla}_k \eta-{A}_k{}^{p\ell}\hat{\nabla}_p \eta-\eta \hat{h}^p_k\hat{h}_p^\ell,\label{eq2.h}\\
\frac{\pt}{\pt t}\sigma_F &= \eta -V^i\hat{\nabla}_i\eta,\label{eq2.sigma}
\end{align}
where the upper indices are lifted using the metric $\hat{g}$, the function $H_{F}=nE_1$ is the anisotropic mean curvature of $M_t$, the tensors $A$ and $Q$ are defined in \S \ref{sec:2-2}, and $V^i=\hat{g}^{i\ell}G(\nu_F)(\partial_\ell X,X)$ stands for the tangential part of the position vector $X$.  In addition, the volume of the evolving domain $\Omega_t$ satisfies
\begin{align}\label{s3:Vol-1}
\frac{d}{dt} \mathrm{Vol}(\Omega_t)&=\int_{M_t}\eta d\mu_F.
\end{align}
\end{lem}

In \cite{Rei76} Reilly derived the following variational formula for the mixed volumes.
\begin{lem}[{\cite{Rei76}}]
Under the equation \eqref{eq2.1},  the mixed volume $V_{n+1-k}(\Omega_t,W_F)$ relative to the Wulff shape $W_F$ evolves by
\begin{equation}\label{s3:evl-Vk}
  \frac d{dt}V_{n+1-k}(\Omega_t,W_F)~=~(n+1-k)\int_{M_t}\eta\: {E}_kd\mu_{F},
\end{equation}
where $k=1,\cdots,n$.
\end{lem}

We now derive the evolution equations of the function $E_k^{1/k}(\kappa)$ and the anisotropic support function $\sigma_F$ along the flow \eqref{flow-1}. We rewrite the flow \eqref{flow-1} as the following form
\begin{equation}\label{s3:flow-2}
    \frac{\partial }{\partial t}X~=~(1-\Psi(\mathcal{W}_F)\sigma_F) \nu_F,
 \end{equation}
where
\begin{equation*}
  \Psi(\mathcal{W}_F)= E_{k}^{1/k}(\kappa),\quad k=2,\cdots,n.
\end{equation*}
Denote by $\dot{\Psi}^{ij}$ and $\ddot{\Psi}^{ij,kl}$ the first and second derivatives of $\Psi$ with respect to the components of $\mathcal{W}_F=(\hat{h}_i^j)$.
\begin{lem}
Along the flow \eqref{s3:flow-2}, the anisotropic support function $\sigma_F$ and the function $\Psi$ of the evolving hypersurface $M_t$ evolve by
\begin{align}\label{s3:evl-u}
  \frac{\partial }{\partial t}\sigma_F =&~\sigma_F\dot{\Psi}^{ij}\left(\hat{\nabla}^j\hat{\nabla}_i\sigma_F+A_{qip}\hat{g}^{jp}\hat{\nabla}^q\sigma_F\right)+\Psi\hat{\nabla}^k\sigma_FG(\nu_F)(X,\partial_kX)\nonumber\\
  &\quad +\sigma_F^2\dot{\Psi}^{ij}\hat{h}_i^k\hat{h}_k^j+1-2\sigma_F\Psi
\end{align}
and
\begin{align}\label{s3:evl-Psi}
  \frac{\partial}{\partial t}\Psi= & \sigma_F\dot{\Psi}^{ij}\left(\hat{\nabla}^j\hat{\nabla}_i\Psi+A_{qip}\hat{g}^{jp}\hat{\nabla}^q\Psi\right)+2\dot{\Psi}^{ij}\hat{\nabla}_i\sigma_F\hat{\nabla}^j\Psi\nonumber\\
  &\quad +\Psi\hat{\nabla}^k\Psi G(\nu_F)(X,\partial_kX)+\Psi^2-\dot{\Psi}^{ij}\hat{h}_i^k\hat{h}_k^j,
\end{align}
where $A_{ijk}$ is the tensor defined in \eqref{s2:A-def}.
\end{lem}
\begin{proof}
(i) First, by \eqref{eq2.sigma}, we see that along the equation \eqref{eq2.1}, the anisotropic support function $\sigma_F$ satisfies
\begin{align}
 \frac{\partial}{\partial t}\sigma_F =& \eta-\hat{\nabla}^k\eta G(\nu_F)(X,\partial_kX).\label{s3:evl-u1}
\end{align}
Using the anisotropic Weingarten equation \eqref{s2:AWeingart}, we also have
\begin{align}\label{s3:dsigm}
 \hat{\nabla}_i\sigma_F =& \hat{h}_i^pG(\nu_F)(\partial_pX,X).
\end{align}
Taking the covariant derivative of \eqref{s3:dsigm} and using the Gauss and Weingarten formulas \eqref{s2:Gauss-for}--\eqref{s2:AWeingart}, we have
\begin{align*}
  \hat{\nabla}_j\hat{\nabla}_i\sigma_F =& \hat{\nabla}_j\hat{h}_i^pG(\nu_F)(\partial_pX,X)+\hat{h}_i^pQ(\nu_F)(\hat{h}_j^\ell\partial_\ell X,\partial_pX,X) \\
  & \quad +\hat{h}_i^pG(\nu_F)\left(-\hat{h}_{jp}\nu_F+\hat{g}^{k\ell}A_{jp\ell}\partial_kX,X\right)+\hat{h}_{ij}\\
  =&~\hat{h}_{ij}-\hat{h}_i^p\hat{h}_{jp}\sigma_F+\left( \hat{\nabla}_j\hat{h}_i^k+\hat{h}_i^p\hat{h}_j^\ell Q_{\ell pq}\hat{g}^{qk}+\hat{h}_i^p\hat{g}^{k\ell}A_{jp\ell}\right)G(\nu_F)(\partial_kX,X)\\
  =&~\hat{h}_{ij}-\hat{h}_i^p\hat{h}_{jp}\sigma_F+\left( \hat{\nabla}^k\hat{h}_{ij}+\hat{h}_j^\ell A_{\ell qi}\hat{g}^{qk}-\hat{h}_q^\ell A_{\ell ji}\hat{g}^{qk}\right)G(\nu_F)(\partial_kX,X)\nonumber\\
  &+\left( \hat{h}_i^p\hat{h}_j^\ell Q_{\ell pq}\hat{g}^{qk}+\hat{h}_i^p\hat{g}^{k\ell}A_{jp\ell}\right)G(\nu_F)(\partial_kX,X)\\
  =&~\hat{h}_{ij}-\hat{h}_i^p\hat{h}_{jp}\sigma_F+ \hat{\nabla}^k\hat{h}_{ij}G(\nu_F)(\partial_kX,X)-A_{\ell ji}\hat{\nabla}^\ell\sigma_F\\
  &\quad +\left(\hat{h}_j^\ell A_{\ell qi}\hat{g}^{qk}+\hat{h}_i^p\hat{h}_j^\ell Q_{\ell pq}\hat{g}^{qk}+\hat{h}_i^p\hat{g}^{k\ell}A_{jp\ell}\right)G(\nu_F)(\partial_kX,X),
\end{align*}
where in the third equality we used the Codazzi equation \eqref{s2:codazzi}.

Regarding the term in the last line of the above equation, we take a local orthonormal frame $\{e_i\}_{i=1}^n$ on $M_t$ such that $(\hat{h}_i^j)=\mathrm{diag}(\kappa_1,\cdots,\kappa_n)$ at the point of computation. Then using $\dot{\Psi}^{ij}=\dot{\psi}^i\delta_{ij}$ (see \eqref{s5:phi-d} below), we can compute directly to get
\begin{equation*}
\dot{\Psi}^{ij}\left(\hat{h}_j^\ell A_{\ell qi}\hat{g}^{qk}+\hat{h}_i^p\hat{h}_j^\ell Q_{\ell pq}\hat{g}^{qk}+\hat{h}_i^p\hat{g}^{k\ell}A_{jp\ell}\right)=0.
\end{equation*}
Therefore we have
\begin{equation}\label{s3:sig_F}
\dot{\Psi}^{ij}(\hat{\nabla}_j\hat{\nabla}_i\sigma_F+A_{\ell ji}\hat{\nabla}^\ell\sigma_F+\sigma_F\hat{h}_i^p\hat{h}_{jp})=\dot{\Psi}^{ij}(\hat{h}_{ij}+ \hat{\nabla}^k\hat{h}_{ij}G(\nu_F)(\partial_kX,X)).
\end{equation}
Then \eqref{s3:evl-u} follows by setting $\eta=1-\Psi \sigma_F$ and using the $1$-homogeneity of $\Psi$.

(ii) By \eqref{eq2.h} and noting that $\eta=1-\Psi \sigma_F$, we have
\begin{align*}
  \frac{\partial}{\partial t}\Psi=& \dot{\Psi}^{ij} \frac{
  \partial}{\partial t}\hat{h}_i^j\\
  = & \dot{\Psi}^{ij} \left(\hat{\nabla}^j\hat{\nabla}_i (\Psi\sigma_F)+{A}_i{}^{pj}\hat{\nabla}_p (\Psi\sigma_F)+(\Psi\sigma_F-1) \hat{h}^p_i\hat{h}_p^j\right)\\
   =& \sigma_F\dot{\Psi}^{ij} \left(\hat{\nabla}^j\hat{\nabla}_i\Psi+{A}_i{}^{pj}\hat{\nabla}_p \Psi\right)+2\dot{\Psi}^{ij}\hat{\nabla}_i\sigma_F\hat{\nabla}^j\Psi\\
   &+\Psi\dot{\Psi}^{ij}\left(\hat{\nabla}^j\hat{\nabla}_i\sigma_F+{A}_i{}^{pj}\hat{\nabla}_p\sigma_F+\sigma_F\hat{h}^p_i\hat{h}_p^j\right)-\dot{\Psi}^{ij}\hat{h}^p_i\hat{h}_p^j.
\end{align*}
Note that the tensor $A$ is symmetric in the first two indices. Substituting \eqref{s3:sig_F} into the above equation and using the $1$-homogeneity of $\Psi$, we obtain the equation \eqref{s3:evl-Psi}.
\end{proof}

\section{$C^0$ and $C^1$ estimates}\label{sec:C01}

In this section, we derive the $C^0$ and $C^1$ estimates of the solution $M_t$ to the flow \eqref{flow-1}.
\subsection{Radial graph parametrization}
We view the Euclidean space as a warped product space  $\R^{n+1}=[0,+\infty)\times \SS^n$ equipped with the metric
\begin{equation*}
g_{\R^{n+1}}=d\rho^2+\rho^2 g_{\SS^n},\quad \rho\in (0,+\infty).
\end{equation*}
Then any star-shaped hypersurface $M$ in $\mathbb{R}^{n+1}$ can be written as a graph over $\SS^n$:
\begin{equation*}
X(\theta)=(\rho(\theta),\theta),\quad \: \theta\in \SS^n.
\end{equation*}
This is also equivalent to that the support function $\sigma=\langle X,\nu\rangle$ is positive everywhere on the hypersurface. Take a local coordinate system $\{\theta^i\}_{i=1}^n$ for $\SS^n$, denote by $e_{ij}$ the components of the round metric $g_{\SS^n}$ on $\SS^n$, and set $\gamma=\log \rho$. Then the unit normal vector field of the hypersurface is given by
\begin{equation}\label{s4.nu}
  \nu=\frac 1{\omega}(1,-\frac{1}{\rho^2}\nabla^{\mathbb{S}}\rho),
\end{equation}
where $\omega=\sqrt{1+|\nabla^\SS \gamma|^2}$. Since the vector field $\rho\partial_{\rho}$ is just the position vector $X$, we see that the support function $\sigma=\rho/\omega$. The first and second fundamental forms of $M$ can be expressed as
\begin{align*}
g_{ij}&=\rho^2(e_{ij}+\gamma_i\gamma_j),\quad g^{ij}=\frac{1}{\rho^2}\left(e^{ij}-\frac{\gamma_i\gamma_j}{\omega^2}\right),\\
h_{ij}&=\frac{\rho}{\omega}(-\gamma_{ij}+\gamma_i\gamma_j+e_{ij}),\\
h^i_j&=\frac{1}{\rho\omega}\left(e^{ik}-\frac{\gamma_i\gamma_k}{\omega^2}\right)(-\gamma_{kj}+\gamma_k\gamma_j+e_{kj}).
\end{align*}
It follows from \eqref{s2:hat-h} that the anisotropic Weingarten matrix $\mathcal{W}_F=(\hat{h}_i^j)$ satisfies
\begin{align}\label{s4.hat-h}
  \hat{h}_i^j =& (A_F(\nu))_i^kh_{k}^j \nonumber\\
  =& \frac{1}{\rho\omega}(A_F(\nu))_i^k\left(e^{j\ell}-\frac{\gamma_j\gamma_{\ell}}{\omega^2}\right)(-\gamma_{\ell k}+\gamma_\ell\gamma_k+e_{k\ell})\nonumber\\
  =& \frac{1}{\rho\omega}(A_F(\nu))_i^k\left(\delta_k^j-\left(e^{j\ell}-\frac{\gamma_j\gamma_{\ell}}{\omega^2}\right)\gamma_{\ell k}\right),
\end{align}
where $\nu$ is evaluated as in \eqref{s4.nu}.

\subsection{Star-shapedness and short time existence}
We assume that the initial hypersurface $M_0$ is strictly convex with the origin contained inside the domain $\Omega_0$ bounded by $M_0$. Then $M_0$ is star-shaped with respect to the origin. It follows that the support function $\sigma$ of $M_0$ is positive,  which implies that the anisotropic support function $\sigma_F=\sigma /F(\nu)$ is positive on $M_0$ as well.
\begin{lem}\label{s4:lem-sigm}
If the initial hypersurface $M_0$ is star-shaped with respect to the origin, then the anisotropic support function $\sigma_F$ of the solution $M_t$ to the flow \eqref{flow-1} is positive for any positive time $t>0$.
\end{lem}
\begin{proof}
Since $\Psi=E_k^{1/k}$ is inverse-concave with respect to its arguments, we have  (see, e.g., \cite[Lemma 5]{And-MZ13})
\begin{equation}\label{s4.psiij}
  \dot{\Psi}^{ij}\hat{h}_i^k\hat{h}_k^j\geq \Psi^2.
\end{equation}
Then the evolution equation \eqref{s3:evl-u} for $\sigma_F$ implies that
\begin{align*}
  \partial_t\sigma_F \geq &~\sigma_F\dot{\Psi}^{ij}\left(\hat{\nabla}^j\hat{\nabla}_i\sigma_F+A_{qip}\hat{g}^{jp}\hat{\nabla}^q\sigma_F\right)+\Psi\hat{\nabla}^k\sigma_FG(\nu_F)(X,\partial_kX)+(\sigma_F\Psi-1)^2\nonumber\\
  \geq &~\sigma_F\dot{\Psi}^{ij}\left(\hat{\nabla}^j\hat{\nabla}_i\sigma_F+A_{qip}\hat{g}^{jp}\hat{\nabla}^q\sigma_F\right) +\Psi\hat{\nabla}^k\sigma_FG(\nu_F)(X,\partial_kX).
\end{align*}
It follows from the parabolic maximum principle that
\begin{equation}\label{s5:sigma_F-lbd}
\sigma_F(y,t)\geq \min_{y\in {M}} \sigma_F(y,0)>0.
\end{equation}
\end{proof}

Since $\sigma=\sigma_F{F}(\nu)$ and ${F}(\nu)>0$, we have $\sigma>0$ on $M_t$ for any time $t>0$. That is, the evolving hypersurface $M_t$ remains to be star-shaped with respect to the origin along the flow \eqref{flow-1}. Thus we can write $M_t$ as the graph of a radial function $\rho(\theta,t)$ over $\mathbb{S}^n$. Since the anisotropic normal vector field is given by $\nu_F=F(\nu)\nu+\nabla^{\mathbb{S}}F(\nu)$ and $\nabla^{\mathbb{S}}{F}(\nu)$ is perpendicular to $\nu$, the flow \eqref{s3:flow-2} is equivalent to
\begin{equation}\label{s5:flow-3}
  \frac{\pt }{\pt t}X=(1-\Psi\sigma_F){F}(\nu)\nu
\end{equation}
up to a smooth reparametrization. It is a well known result that the flow \eqref{s5:flow-3} for star-shaped hypersurfaces is equivalent to the following scalar parabolic PDE
\begin{equation}\label{s5:rho-t}
  \frac{\pt }{\pt t}\rho=\omega (1-\Psi\sigma_F){F}(\nu)
\end{equation}
of the radial function $\rho(\theta,t)$ on $\mathbb{S}^n$. By the facts $\sigma=\sigma_F{F}(\nu)$ and $\sigma=\rho/\omega$, the equation \eqref{s5:rho-t} can be rewritten as the following fully nonlinear parabolic equation
\begin{align}\label{eq3.gamma}
\frac{\pt}{\pt t} \gamma&=\frac{\omega }{\rho}{F(\nu)}-\frac{1}{\rho \omega}\Psi\left((A_F(\nu))_i^k \left(\delta^j_k-\left(e^{j\ell}-\frac{\gamma_j\gamma_\ell}{\omega^2}\right)\gamma_{k\ell}\right)\right)
\end{align}
of the function $\gamma(\theta,t)=\log\rho(\theta,t)$ on $\mathbb{S}^n$, where $\nu$ is evaluated as in \eqref{s4.nu}.

Since the initial hypersurface $M_0$ is strictly convex, the short time existence of \eqref{eq3.gamma} is standard by using the implicit function theorem.

\begin{lem}
Let $M_0$ be a smooth, closed and strictly convex hypersurface in $\mathbb{R}^{n+1}$. Then there exists a unique smooth solution $M_t$ of the flow \eqref{flow-1} starting from $M_0$ on some maximal time interval $[0,T)$ with $T\leq \infty$. The hypersurface $M_t$ is star-shaped for any $t\in [0,T)$.
\end{lem}

\subsection{$C^0$ estimate} \label{sec:C0}
Since $M_0$ is closed and strictly convex, and it encloses the origin, we see that it is star-shaped with respect to the origin. Then there exist two constants $0<r<R$ such that
\begin{equation*}
rW_F\subset \Omega_0\subset RW_F.
\end{equation*}
Due to the fact that the anisotropic principal curvatures of $\Sigma_F$ are equal to $1$ and the anisotropic support function of $\Sigma_F$ is a constant function $\sigma_F=1$, both the hypersurfaces $r\Sigma_F$ and $R\Sigma_F$ are fixed along the flow \eqref{flow-1}. Moreover, the velocity of the flow \eqref{flow-1} is defined pointwisely. Applying the comparison principle, we get
\begin{equation*}
rW_F\subset \Omega_t \subset RW_F,\text{ for }t\in [0,T).
\end{equation*}
So we have the $C^0$ estimate.

\subsection{$C^1$ estimate}\label{sec:C1}
The $C^1$ estimate follows from the estimate \eqref{s5:sigma_F-lbd} on the anisotropic support function and the $C^0$ estimate. Recall that
\begin{equation}\label{eq3.sigma2}
\sigma_F=\frac{\sigma}{{F}(\nu)}=\frac{\rho}{\omega {F}(\nu)}=\frac{\rho}{\sqrt{1+|\nabla^\SS \gamma|^2} {F}(\nu)}.
\end{equation}
By the estimate \eqref{s5:sigma_F-lbd},
\begin{equation*}
\sigma_F(y,t)\geq \min_{y\in {M}} \sigma_F(y,0),
\end{equation*}
which combined with \eqref{eq3.sigma2} and the $C^0$ estimate implies that
\begin{equation*}
 |\nabla^\SS \gamma|(\cdot,t)\leq C,
\end{equation*}
where $C$ depends on ${F}$, $r$, $R$ and $|\nabla^\SS \gamma|(\cdot,0)$.

\section{Anisotropic Gauss map parametrization}\label{sec:Gauss}
To derive the curvature estimate of the solution $M_t$, we will employ the anisotropic Gauss map parametrization of $M_t$.  In this section, we first review some basic results on this parametrization.

Given a smooth, closed and strictly convex hypersurface $M$ in $\mathbb{R}^{n+1}$, the Gauss map is defined as $\nu: M\to \mathbb{S}^n$ which maps the point $x\in M$ to the outward unit normal $\nu\in \mathbb{S}^n$ at this point. The Gauss map of a smooth closed strictly convex hypersurface is a nondegenerate diffeomorphism between $M$ and $\mathbb{S}^n$. Since the Wulff shape $W_F$ is also strictly convex, we can define a map from $\mathbb{S}^n$ to $\Sigma_F=\partial W_F$ which maps $\nu$ to $\nu_F=DF(\nu)$. Then the anisotropic Gauss map $\nu_F:M\to \Sigma_F$ is the composition of the above two maps and satisfies $$\nu_F(x)=DF(\nu(x))$$ for each point $x\in M$. It follows that the anisotropic Gauss map is a nondegenerate diffeomorphism as well. Thus we can use it to reparametrize $M$:
\begin{equation*}
  X:\Sigma_F\to M\subset \mathbb{R}^{n+1},\quad X(z)=X(\nu_F^{-1}(z)),\quad z\in \Sigma_F.
\end{equation*}
The anisotropic support function of $M$ is then defined as a function on $\Sigma_F$ by $s(z)=G(z)(z,X(z))$ for $z\in \Sigma_F$. Comparing it with the definition \eqref{s2:spt-def} of $\sigma_F$, we see that $\sigma_F(x)=s(\nu_F(x))$.

Let $\bar{g}$ and $\bar{\nabla}$ denote the induced metric and its Levi-Civita connection on $\Sigma_F$ from $(\mathbb{R}^{n+1},G)$ respectively. The anisotropic principal curvatures $\kappa=(\kappa_1,\cdots,\kappa_n)$ of a strictly convex hypersurface $M$ in $\mathbb{R}^{n+1}$ are related to its anisotropic principal radii of curvature $\tau=(\tau_1,\cdots,\tau_n)$ by
\begin{equation*}
  \kappa_i=1/{\tau_i},\quad i=1,\cdots,n.
\end{equation*}
It has been shown that the anisotropic principal radii of curvature $\tau=(\tau_1,\cdots,\tau_n)$ of $M$ are eigenvalues of the following matrix on the Wulff shape $\Sigma_F$ (see \cite{Xia13})
\begin{equation}\label{s6:tau-def}
  \tau_{ij}[s]=\bar{\nabla}_i\bar{\nabla}_js+\bar{g}_{ij}s-\frac 12Q_{ijk}\bar{\nabla}_ks,
\end{equation}
where $s:\Sigma_F\to \mathbb{R}$ is the anisotropic support function on the hypersurface $M$.

\begin{lem}\label{s6:lem1}
We have the following Codazzi and Simons type equations for $\tau_{ij}$:
\begin{equation}\label{s6:Codaz}
  \bar{\nabla}_j\tau_{k\ell}+\frac 12Q_{k\ell p}\tau_{jp}~=~\bar{\nabla}_k\tau_{j\ell}+\frac 12Q_{j\ell p}\tau_{kp}
\end{equation}
and
\begin{align}\label{s6:Simons}
 \bar{\nabla}_{(i} \bar{\nabla}_{j)}\tau_{k\ell} =&\bar{\nabla}_{(k}\bar{\nabla}_{\ell)}\tau_{ji}+\frac 12Q_{jiq}\bar{\nabla}_q\tau_{k\ell}-\frac 12Q_{k\ell p}\bar{\nabla}_p\tau_{ij}\nonumber\\
 & +\bar{g}_{ij}\tau_{k\ell}-\bar{g}_{k\ell}\tau_{ij} +Q*Q*\tau+\bar{\nabla} Q*\tau,
\end{align}
where $*$ denotes contraction of tensors using the metric on $\Sigma_F$, and the bracket $(\,,\,)$ on the indices denotes the symmetrization,meaning e.g.
$$
\bar{\nabla}_{(i} \bar{\nabla}_{j)}\tau_{k\ell}=(\bar{\nabla}_{i} \bar{\nabla}_{j}\tau_{k\ell}+\bar{\nabla}_{j} \bar{\nabla}_{i}\tau_{k\ell})/2.
$$
\end{lem}
The proof of \eqref{s6:Codaz} is by taking covariant derivatives of \eqref{s6:tau-def}, and using the Gauss equation \eqref{s2:gauss-2} and the Ricci identity. See the proof of Lemma 5.2 in \cite{Xia13}. Taking the covariant derivative of \eqref{s6:Codaz}  and using the Gauss equation \eqref{s2:gauss-2} and the Ricci identity, we can obtain \eqref{s6:Simons} after rearranging the terms. We include a detailed calculation in Appendix \ref{sec:ap1}.

Let $\phi$ denote the dual function of $E_{k}^{1/k}$, which is defined as
\begin{align*}
  \phi(x_1,\cdots,x_n)=&\left(E_{k}(x_1^{-1},\cdots,x_n^{-1})\right)^{-1/k}\\
  =&\left(\frac {E_{n}(x)}{E_{n-k}(x)}\right)^{1/k}
\end{align*}
for $x\in \Gamma_+=\{x\in \mathbb{R}^n: x_i>0, i=1,\cdots,n\}$. We write $\Phi(\tau_{ij})=\phi(\tau)$ as a smooth symmetric function of the matrix $(\tau_{ij})$, which is viewed as the function $\phi$ evaluated at the eigenvalues $\tau=(\tau_1,\cdots,\tau_n)$ of $\tau_{ij}$. As before, we denote by $\dot{\Phi}^{k\ell}$ and $\ddot{\Phi}^{k\ell,pq}$ the derivatives of $\Phi$ with respect to its arguments. We review the following basic properties on the symmetric functions $\Phi$ and $\phi$, which will be used in the curvature estimate later. See \cite{And-MZ13} for more properties.
\begin{lem}
Let $\mathrm{Sym}(n)$ be the set of $n\times n$ symmetric matrices, and let $\Phi(A)=\phi(\tau)$, where $\phi$ is a smooth symmetric function and $A$ is in $\mathrm{Sym}(n)$ with eigenvalues $\tau=(\tau_1,\cdots,\tau_n)$. If $A$ is diagonal, the first derivatives of $\Phi(A)$ and $\phi(\tau)$ with respect to their arguments are related by the following equation
\begin{equation}\label{s5:phi-d}
  \dot{\Phi}^{ij}(A)=\dot{\phi}^i(\tau)\delta_i^j.
\end{equation}
The second derivative of $\Phi$ in the direction $B\in \mathrm{Sym}(n)$ satisfies
\begin{equation}\label{s5:phi-d2}
  \ddot{\Phi}^{ij,k\ell}(A)B_{ij}B_{k\ell}=\ddot{\phi}^{k\ell}(\tau)B_{kk}B_{\ell\ell}+2\sum_{k<\ell}\frac{\dot{\phi}^k(\tau)-\dot{\phi}^\ell(\tau)}{\tau_k-\tau_\ell}(B_{k\ell})^2.
\end{equation}
This formula makes sense as a limit in the case of any repeated values of $\tau_k$.
\end{lem}

\begin{lem}
The function
\begin{equation*}
  \phi(\tau)=\left(\frac {E_{n}(\tau)}{E_{n-k}(\tau)}\right)^{1/k}
\end{equation*}
is concave with respect to $\tau$, and  we have
\begin{equation}\label{s5:phi-p1}
  \dot{\Phi}^{k\ell}\bar{g}_{k\ell}=\sum_k\dot{\phi}^k\geq 1
\end{equation}
and
\begin{equation}\label{s5:phi-p2}
  (\ddot{\phi}^{ij})\leq 0,\qquad (\dot{\phi}^k(\tau)-\dot{\phi}^\ell(\tau))(\tau_k-\tau_\ell)\leq 0,\quad \forall k\neq \ell.
\end{equation}
\end{lem}

Since the initial hypersurface $M_0$ is strictly convex,  it can be shown that the solution to the flow \eqref{s3:flow-2} is given, up to a time-dependent diffeomorphism,  by solving the scalar parabolic partial differential equation on the Wulff shape $\Sigma_F$
 \begin{equation}\label{s6:flow-gauss}
   \frac{\partial }{\partial t}s(z,t)=1-\frac{s(z,t)}{\Phi(\tau_{ij}[s(z,t)])}
 \end{equation}
for the anisotropic support function $s$ of $M_t$. In the rest of this section, we derive the evolution equations of $s, |\bar{\nabla}s|^2$ and $\tau_{ij}[s]$ along the flow \eqref{s6:flow-gauss}.
\begin{lem}
The evolution equation \eqref{s6:flow-gauss} for the anisotropic support function $s$ has the following equivalent form:
\begin{align}\label{s6:s-pde}
 & \frac{\partial }{\partial t}s- s\Phi^{-2}\dot{\Phi}^{k\ell}\left(\bar{\nabla}_k\bar{\nabla}_\ell s-\frac 12Q_{k\ell p}\bar{\nabla}_ps\right)= (1-s{\Phi}^{-1})^2+s^2\Phi^{-2}\left(\sum_k\dot{\phi}^k-1\right).
\end{align}
\end{lem}
\begin{proof}
Since $\Phi$ is homogeneous of degree $1$ with respect to $\tau_{ij}$, we have
\begin{equation*}
\Phi= \dot{\Phi}^{k\ell}\tau_{k\ell}= \dot{\Phi}^{k\ell}\left(\bar{\nabla}_k\bar{\nabla}_\ell s+\bar{g}_{k\ell}s-\frac 12Q_{k\ell p}\bar{\nabla}_ps\right).
\end{equation*}
This implies that
\begin{align*}
\frac{\partial }{\partial t}s&- s\Phi^{-2}\dot{\Phi}^{k\ell}\left(\bar{\nabla}_k\bar{\nabla}_\ell s-\frac 12Q_{k\ell p}\bar{\nabla}_ps\right) \\
 = & 1-2 s{\Phi}^{-1}+s^2\Phi^{-2}\dot{\Phi}^{k\ell}\bar{g}_{k\ell}\\
 =& (1-s{\Phi}^{-1})^2+s^2\Phi^{-2}\left(\dot{\Phi}^{k\ell}\bar{g}_{k\ell}-1\right).
\end{align*}
Then \eqref{s6:s-pde} follows from the relation \eqref{s5:phi-p1}.
\end{proof}

\begin{lem}
Under the flow \eqref{s6:flow-gauss}, the squared norm $|\bar{\nabla}s|^2$ of the gradient of the anisotropic support function $s$ evolves according to
\begin{align}\label{s6:ds-pde}
 & \frac{\partial }{\partial t} |\bar{\nabla}s|^2 -s\Phi^{-2}\dot{\Phi}^{k\ell}\left(\bar{\nabla}_k\bar{\nabla}_\ell|\bar{\nabla}s|^2 -\frac 12Q_{k\ell p}\bar{\nabla}_p|\bar{\nabla}s|^2\right)\nonumber\\
 = & ~ -2|\bar{\nabla}s|^2\Phi^{-1}+s\Phi^{-2}\dot{\Phi}^{k\ell}\biggl(-2\tau_{ik}\tau_{i\ell}-2s^2\bar{g}_{k\ell}+4s\tau_{k\ell}-2\tau_{ik}Q_{i\ell p}s_p +2sQ_{k\ell p}s_p\biggr)\nonumber\\
 &\quad +s\Phi^{-2}\dot{\Phi}^{k\ell}\biggl(2s_ks_\ell -\frac 12(2Q_{i\ell m}Q_{mkp}-Q_{ipm}Q_{mk\ell})s_is_p-\bar{\nabla}_iQ_{k\ell p}s_ps_i\biggr).
\end{align}
\end{lem}
\begin{proof}
The evolution of $|\bar{\nabla}s|^2$ along the equation \eqref{s6:flow-gauss} can be computed as follows:
\begin{align*}
  \frac{\partial }{\partial t} |\bar{\nabla}s|^2=&~ 2s_i\bar{\nabla}_i(\partial_ts) =  2s_i\bar{\nabla}_i(1-s\Phi^{-1})\\
  =&~-2|\bar{\nabla}s|^2\Phi^{-1}+2s\Phi^{-2}s_i\dot{\Phi}^{k\ell}\bar{\nabla}_i\tau_{k\ell}\\
  =&~-2|\bar{\nabla}s|^2\Phi^{-1}+2s\Phi^{-2}s_i\dot{\Phi}^{k\ell}\biggl(\bar{\nabla}_i\bar{\nabla}_k\bar{\nabla}_\ell s+\bar{\nabla}_is\bar{g}_{k\ell}\\
  &\quad -\frac 12Q_{k\ell p}\bar{\nabla}_i\bar{\nabla}_ps-\frac 12\bar{\nabla}_iQ_{k\ell p}\bar{\nabla}_ps\biggr).
\end{align*}
By the Ricci identity and \eqref{s2:gauss-2}, we get
\begin{align*}
  \bar{\nabla}_k\bar{\nabla}_\ell|\bar{\nabla}s|^2 =& 2s_{ik}s_{i\ell}+2s_i\bar{\nabla}_k\bar{\nabla}_\ell\bar{\nabla}_is \\
  =&2s_{ik}s_{i\ell}+2s_i(\bar{\nabla}_i\bar{\nabla}_k\bar{\nabla}_\ell s+\bar{R}_{ki\ell p}\bar{\nabla}_ps) \\
  =& 2s_{ik}s_{i\ell}+2s_i\bar{\nabla}_i\bar{\nabla}_k\bar{\nabla}_\ell s+2|\bar{\nabla}s|^2\bar{g}_{k\ell}-2s_ks_\ell \\
  &\quad +\frac 12(Q_{i\ell m}Q_{mkp}-Q_{ipm}Q_{mk\ell})s_is_p.
\end{align*}
Combining the above two equations yields
\begin{align*}
 & \frac{\partial }{\partial t} |\bar{\nabla}s|^2 -s\Phi^{-2}\dot{\Phi}^{k\ell}\left(\bar{\nabla}_k\bar{\nabla}_\ell|\bar{\nabla}s|^2 -\frac 12Q_{k\ell p}\bar{\nabla}_p|\bar{\nabla}s|^2\right)\\
 =&~-2|\bar{\nabla}s|^2\Phi^{-1}+s\Phi^{-2}\dot{\Phi}^{k\ell}\biggl(-2s_{ik}s_{i\ell}+2s_ks_\ell\\
  &\quad -\frac 12(Q_{i\ell m}Q_{mkp}-Q_{ipm}Q_{mk\ell})s_is_p-\bar{\nabla}_iQ_{k\ell p}s_ps_i\biggr)\\
= & ~ -2|\bar{\nabla}s|^2\Phi^{-1}+s\Phi^{-2}\dot{\Phi}^{k\ell}\biggl(-2\tau_{ik}\tau_{i\ell}-2s^2\bar{g}_{k\ell}+4s\tau_{k\ell}-2\tau_{ik}Q_{i\ell p}s_p +2sQ_{k\ell p}s_p\biggr)\\
 &\quad +s\Phi^{-2}\dot{\Phi}^{k\ell}\biggl(2s_ks_\ell -\frac 12(2Q_{i\ell m}Q_{mkp}-Q_{ipm}Q_{mk\ell})s_is_p-\bar{\nabla}_iQ_{k\ell p}s_ps_i\biggr).
\end{align*}
\end{proof}

\begin{lem}\label{s6:lem-tau}
Under the flow \eqref{s6:flow-gauss}, the tensor $\tau_{ij}[s]$ evolves by
\begin{align}\label{s6:flow-tau}
 & \frac{\partial}{\partial t}\tau_{ij}-s\Phi^{-2}\dot{\Phi}^{k\ell} \left(\bar{\nabla}_k\bar{\nabla}_\ell\tau_{ij}-\frac 12Q_{k\ell p}\bar{\nabla}_p\tau_{ij}\right)\nonumber\\
 =& s\Phi^{-2}\ddot{\Phi}^{k\ell,pq}\bar{\nabla}_i\tau_{k\ell}\bar{\nabla}_j\tau_{pq}+s\Phi^{-2}\dot{\Phi}*Q*Q*\tau+s\Phi^{-2}\dot{\Phi}*\bar{\nabla}Q*\tau\nonumber\\ & -\left(\Phi^{-1}+s\Phi^{-2}\dot{\Phi}^{k\ell}\bar{g}_{k\ell}\right)\tau_{ij}+\bar{g}_{ij}(1+s\Phi^{-1})\nonumber\\
&\quad +2\Phi^{-2}\bar{\nabla}_is\bar{\nabla}_j\Phi-2s\Phi^{-3}\bar{\nabla}_i\Phi\bar{\nabla}_j\Phi.
\end{align}
\end{lem}
\begin{proof}
Taking the time derivative of \eqref{s6:tau-def} and using \eqref{s6:flow-gauss}, we have
\begin{align}\label{s6:tau-1}
  \frac{\partial}{\partial t}\tau_{ij}=&\bar{\nabla}_i\bar{\nabla}_j \frac{\partial}{\partial t}s+\bar{g}_{ij} \frac{\partial}{\partial t}s-\frac 12Q_{ijk}\bar{\nabla}_k \frac{\partial}{\partial t}s\nonumber\\
  =&-\bar{\nabla}_i\bar{\nabla}_j\frac{s}{\Phi}+\bar{g}_{ij}(1-\frac{s}{\Phi})+\frac 12Q_{ijk}\bar{\nabla}_k\frac{s}{\Phi}\nonumber\\
=& -\Phi^{-1}\bar{\nabla}_i\bar{\nabla}_js+s\Phi^{-2}\bar{\nabla}_i\bar{\nabla}_j\Phi+2\Phi^{-2}\bar{\nabla}_is\bar{\nabla}_j\Phi-2s\Phi^{-3}\bar{\nabla}_i\Phi\bar{\nabla}_j\Phi\nonumber\\
&\quad +\bar{g}_{ij}(1-\frac{s}{\Phi})+\frac 12Q_{ijk}\left(\Phi^{-1}\bar{\nabla}_ks-s\Phi^{-2}\bar{\nabla}_k\Phi\right)\nonumber\\
=& -\Phi^{-1}\tau_{ij}+s\Phi^{-2}\bar{\nabla}_i\bar{\nabla}_j\Phi+2\Phi^{-2}\bar{\nabla}_is\bar{\nabla}_j\Phi-2s\Phi^{-3}\bar{\nabla}_i\Phi\bar{\nabla}_j\Phi\nonumber\\
&\quad +\bar{g}_{ij}-\frac 12s\Phi^{-2}Q_{ijk}\bar{\nabla}_k\Phi,
\end{align}
where  in the last equality we used the definition \eqref{s6:tau-def} of $\tau_{ij}$ again. Using the Simons identity \eqref{s6:Simons}, the second term on the right-hand side of \eqref{s6:tau-1} can be expressed as
\begin{align}\label{s6:tau-2}
 \bar{\nabla}_i\bar{\nabla}_j\Phi=& \dot{\Phi}^{k\ell}\bar{\nabla}_i\bar{\nabla}_j\tau_{k\ell}+\ddot{\Phi}^{k\ell,pq}\bar{\nabla}_i\tau_{k\ell}\bar{\nabla}_j\tau_{pq}\nonumber\\
 =&\dot{\Phi}^{k\ell} \bar{\nabla}_k\bar{\nabla}_\ell\tau_{ij}+\frac 12Q_{jiq}\bar{\nabla}_q\Phi-\frac 12\dot{\Phi}^{k\ell}Q_{k\ell p}\bar{\nabla}_p\tau_{ij}+\Phi\bar{g}_{ij}-\dot{\Phi}^{k\ell}\bar{g}_{k\ell}\tau_{ij} \nonumber\\
 & +\dot{\Phi}*Q*Q*\tau+\dot{\Phi}*\bar{\nabla}Q*\tau+\ddot{\Phi}^{k\ell,pq}\bar{\nabla}_i\tau_{k\ell}\bar{\nabla}_j\tau_{pq}.
\end{align}
Substituting \eqref{s6:tau-2} into \eqref{s6:tau-1} and rearranging the terms, we obtain \eqref{s6:flow-tau}.
\end{proof}

\section{Curvature estimate}\label{sec:C2}

In this section, we derive uniform upper and lower bounds on the anisotropic principal curvatures. We first show that the function $E_k(\kappa)$ is uniformly bounded from above.
\begin{prop}\label{s7:lem-Psi}
Assume that the initial hypersurface $M_0$ is smooth, closed and strictly convex.  Then along the flow \eqref{flow-1} we have a uniform upper bound on $E_{k}(\kappa)$.
\end{prop}
\begin{proof}
Since $\Psi={E_{k}^{1/k}(\kappa)}$ is inverse-concave, we have $\Psi^2\leq \dot{\Psi}^{ij}\hat{h}_i^k\hat{h}_k^j$. So the evolution equation \eqref{s3:evl-Psi} of $\Psi$ satisfies
\begin{align}\label{s7:evl-Psi}
  \frac{\partial}{\partial t}\Psi\leq  & ~ \sigma_F\dot{\Psi}^{ij}\left(\hat{\nabla}^j\hat{\nabla}_i\Psi+A_{piq}\hat{g}^{jq}\hat{\nabla}^p\Psi\right)+2\dot{\Psi}^{ij}\hat{\nabla}_i\sigma_F\hat{\nabla}^j\Psi\nonumber\\
  &\quad +\Psi\hat{\nabla}^k\Psi G(\nu_F)(X,\partial_kX).
\end{align}
Then the parabolic maximum principle implies that  $\max_{M_t}\Psi$ is non-increasing in time. Therefore $\max_{M_t}\Psi\leq \max_{M_0}\Psi$.
\end{proof}

Next, we derive the uniform positive lower bound on $\kappa_i$ of $M_t$ for $t>0$.
\begin{prop}\label{s7:prop-tau}
Along the flow \eqref{flow-1}, the anisotropic principal curvatures $\kappa_i$ of the solution $M_t$ have a uniform positive lower bound for $t>0$.
\end{prop}
\begin{proof}
To prove this estimate, we employ the anisotropic Gauss map parametrization of the flow \eqref{flow-1} as described in \S \ref{sec:Gauss}. The flow \eqref{flow-1} is equivalent to the following scalar parabolic equation of the anisotropic support function $s(z,t)$:
  \begin{equation}\label{s7:flow-s}
   \frac{\partial }{\partial t}s(z,t)=1-\frac{s(z,t)}{\Phi(\tau_{ij})},\quad (z,t)\in \Sigma_F\times[0,T),
 \end{equation}
where
\begin{equation*}
  \Phi(\tau_{ij})=\left(\frac{E_{n}(\tau)}{E_{n-k}(\tau)}\right)^{1/k}
\end{equation*}
is the dual function of $\Psi$ and $\tau_{ij}=\tau_{ij}[s]$ is the matrix defined in \eqref{s6:tau-def}. The eigenvalues $\tau=(\tau_1,\cdots,\tau_n)$ of $\tau_{ij}$ are equal to the reciprocal of the anisotropic principal curvatures $\kappa=(\kappa_1,\cdots,\kappa_n)$. Therefore, in order to estimate the lower bound of $\kappa$, it suffices to estimate the upper bound of the eigenvalues of $\tau_{ij}$.

We choose the orthonormal frame such that $(\tau_{ij})$ is diagonal at the point we are considering. Suppose that $e_1$ is the direction where the largest eigenvalue of $(\tau_{ij})$ occurs. By the evolution equation \eqref{s6:flow-tau} of $\tau_{ij}$, we have
\begin{align}\label{s7:flow-tau2}
  &\frac{\partial}{\partial t}\tau_{11}- s\Phi^{-2}\left(\dot{\Phi}^{k\ell} \bar{\nabla}_k\bar{\nabla}_\ell\tau_{11}-\frac 12\dot{\Phi}^{k\ell}Q_{k\ell p}\bar{\nabla}_p\tau_{11}\right)\nonumber\\
\leq & ~ s\Phi^{-2}\ddot{\Phi}^{k\ell,pq}\bar{\nabla}_1\tau_{k\ell}\bar{\nabla}_1\tau_{pq}+Cs\Phi^{-2}\dot{\Phi}^{k\ell}\bar{g}_{k\ell}\tau_{11}\nonumber\\
 &-\left(\Phi^{-1}+s\Phi^{-2}\dot{\Phi}^{k\ell}\bar{g}_{k\ell}\right)\tau_{11} +(1+s\Phi^{-1})\nonumber\\
 & +2\Phi^{-2}\bar{\nabla}_1s\bar{\nabla}_1\Phi-2s\Phi^{-3}|\bar{\nabla}_1\Phi|^2,
\end{align}
where $C=C(Q,\bar{\nabla}Q)$ is a constant depending only on $Q$ and $\bar{\nabla}Q$.  Since $\phi(\tau)$ is increasing in each $\tau_i$, we have
\begin{equation}\label{s6.1}
  \tau_{11}=\tau_{11}\phi(1)\geq \phi(\tau)=\Phi(\tau_{ij}).
\end{equation}
Using the fact
\begin{equation}\label{s6.2}
  \dot{\Phi}^{k\ell}\bar{g}_{k\ell}=\sum_k\dot{\phi}^k\geq 1,
\end{equation}
we see that $(1+s\Phi^{-1})$ can be cancelled out by $-\left(\Phi^{-1}+s\Phi^{-2}\dot{\Phi}^{k\ell}\bar{g}_{k\ell}\right)\tau_{11}$. Therefore, the second line on the right-hand side of \eqref{s7:flow-tau2} is non-positive. On the other hand, the last line of \eqref{s7:flow-tau2} can be estimated as
\begin{align}\label{s6.3}
 &2\Phi^{-2}\bar{\nabla}_1s\bar{\nabla}_1\Phi-2s\Phi^{-3}|\bar{\nabla}_1\Phi|^2\nonumber\\
 =& -2s\Phi^{-3}\left(\bar{\nabla}_1\Phi-\frac{1}{2}\Phi\bar{\nabla}_1\log s\right)^2
 +\frac 1{2s\Phi}|\bar{\nabla}_1s|^2\nonumber\\
\leq  &\frac 1{2s\Phi}|\bar{\nabla}_1s|^2.
\end{align}
By the $C^0$ and $C^1$ estimates of $s$ and the facts \eqref{s6.1}--\eqref{s6.2}, the last term of \eqref{s6.3} can be absorbed by the term $Cs\Phi^{-2}\dot{\Phi}^{k\ell}\bar{g}_{k\ell}\tau_{11}$, where $C$ depends on $F$ and $M_0$. Therefore, we arrive at
\begin{align}\label{s7:flow-tau3}
  \frac{\partial}{\partial t}\tau_{11}\leq & s\Phi^{-2}\left(\dot{\Phi}^{k\ell} \bar{\nabla}_k\bar{\nabla}_\ell\tau_{11}-\frac 12\dot{\Phi}^{k\ell}Q_{k\ell p}\bar{\nabla}_p\tau_{11}\right)\nonumber\\
 &+s\Phi^{-2}\ddot{\Phi}^{k\ell,pq}\bar{\nabla}_1\tau_{k\ell}\bar{\nabla}_1\tau_{pq}+ Cs\Phi^{-2}\left(\sum_k\dot{\phi}^k\right)\tau_{11},
\end{align}
where $C$ depends on $F$ and $M_0$. Since $\sum_k\dot{\phi}^k$ may not be bounded, we can not apply the parabolic maximum principle directly to conclude that $\tau_{11}$ is bounded from above on a finite time interval.

To overcome this problem, we define
\begin{equation*}
  \zeta= \sup\{\tau_{ij}\xi^i\xi^j:~|\xi|=1\}
\end{equation*}
and consider the auxiliary function
\begin{equation}\label{s6:omeg-def}
  \omega(z,t) =\log\zeta(z,t)-\alpha s(z,t)+\beta |\bar{\nabla}s|^2
\end{equation}
on $\Sigma_F\times [0,T)$, where $\alpha, \beta>0$ are positive constants to be determined later. We consider a point $(z_1,t_1)$ where a new maximum of the function $\omega$ is achieved, i.e.,  $\omega(z_1,t_1)=\max_{\Sigma_F\times [0,t_1]}\omega$ and $\omega(z,t)<\omega(z_1,t_1)$ for $t<t_1$. By rotation of local orthonormal frame, we assume that $\xi=e_1$ and $(\tau_{ij})=\mathrm{diag}(\tau_1,\cdots,\tau_n)$ is diagonal at $(z_1,t_1)$. Firstly, at $(z_1,t_1)$ we have $\zeta=\tau_{11}$ and
\begin{align*}
  \frac{\partial }{\partial t}\omega =& \frac 1{\tau_{11}}\frac{\partial }{\partial t}\tau_{11} -\alpha \frac{\partial }{\partial t}s+\beta\frac{\partial }{\partial t}|\bar{\nabla}s|^2,\\
  \bar{\nabla}_\ell\omega =& \frac 1{\tau_{11}}\bar{\nabla}_\ell\tau_{11}-\alpha \bar{\nabla}_\ell s+\beta\bar{\nabla}_\ell|\bar{\nabla}s|^2,\\
  \bar{\nabla}_k\bar{\nabla}_\ell\omega =&\frac 1{\tau_{11}}\bar{\nabla}_k\bar{\nabla}_\ell\tau_{11}-\alpha \bar{\nabla}_k\bar{\nabla}_\ell s+\beta\bar{\nabla}_k\bar{\nabla}_\ell|\bar{\nabla}s|^2-\frac 1{(\tau_{11})^2}\bar{\nabla}_k\tau_{11}\bar{\nabla}_\ell\tau_{11}.
\end{align*}
Combining these equations with \eqref{s6:s-pde}, \eqref{s6:ds-pde} and \eqref{s7:flow-tau3} gives
\begin{align}\label{s7:evl-S}
   &\frac{\partial }{\partial t}\omega -s\Phi^{-2}\dot{\Phi}^{k\ell}\left(\bar{\nabla}_k\bar{\nabla}_\ell\omega-\frac 12Q_{k\ell p}\bar{\nabla}_p\omega\right)\nonumber\\
   =&\frac 1{\tau_{11}}\left(\frac{\partial }{\partial t}\tau_{11} -s\Phi^{-2}\dot{\Phi}^{k\ell}\left(\bar{\nabla}_k\bar{\nabla}_\ell\tau_{11}-\frac 12Q_{k\ell p}\bar{\nabla}_p\tau_{11}\right)\right)\nonumber\\
   & -\alpha \left(\frac{\partial }{\partial t}s -s\Phi^{-2}\dot{\Phi}^{k\ell}\left(\bar{\nabla}_k\bar{\nabla}_\ell s-\frac 12Q_{k\ell p}\bar{\nabla}_ps\right)\right)\nonumber\\
   &+\beta \left(\frac{\partial }{\partial t}|\bar{\nabla}s|^2 -s\Phi^{-2}\dot{\Phi}^{k\ell}\left(\bar{\nabla}_k\bar{\nabla}_\ell|\bar{\nabla}s|^2-\frac 12Q_{k\ell p}\bar{\nabla}_p|\bar{\nabla}s|^2\right)\right)\nonumber\\
   &+\frac 1{(\tau_{11})^2}s\Phi^{-2}\dot{\Phi}^{k\ell}\bar{\nabla}_k\tau_{11}\bar{\nabla}_\ell\tau_{11}\nonumber\\
  \leq &\frac {s\Phi^{-2}}{\tau_{11}}\left(\ddot{\Phi}^{k\ell,pq}\bar{\nabla}_1\tau_{k\ell}\bar{\nabla}_1\tau_{pq}+C\sum_k\dot{\phi}^{k}\tau_{11}\right)\nonumber\\
  &-\alpha\left(1- s{\Phi}^{-1}\right)^2-\alpha s^2\Phi^{-2}\left(\sum_k\dot{\phi}^{k}-1\right) +\frac 1{(\tau_{11})^2}s\Phi^{-2}\dot{\Phi}^{k\ell}\bar{\nabla}_k\tau_{11}\bar{\nabla}_\ell\tau_{11}\nonumber\\
  &+\beta\biggl(-2|\bar{\nabla}s|^2\Phi^{-1}+s\Phi^{-2}\dot{\Phi}^{k\ell}\biggl(-2\tau_{ik}\tau_{i\ell}-2s^2\bar{g}_{k\ell}+4s\tau_{k\ell}-2\tau_{ik}Q_{i\ell p}s_p +2sQ_{k\ell p}s_p\biggr)\nonumber\\
 &\quad +s\Phi^{-2}\dot{\Phi}^{k\ell}\biggl(2s_ks_\ell -\frac 12(2Q_{i\ell m}Q_{mkp}-Q_{ipm}Q_{mk\ell})s_is_p-\bar{\nabla}_iQ_{k\ell p}s_ps_i\biggr)\biggr)\nonumber\\
 \leq & s\Phi^{-2}\left(\frac 1{\tau_{11}}\ddot{\Phi}^{k\ell,pq}\bar{\nabla}_1\tau_{k\ell}\bar{\nabla}_1\tau_{pq}+\frac 1{(\tau_{11})^2}\dot{\Phi}^{k\ell}\bar{\nabla}_k\tau_{11}\bar{\nabla}_\ell\tau_{11}\right)\nonumber\\
  &-\alpha\left(1- s{\Phi}^{-1}\right)^2-\alpha s^2\Phi^{-2}\left(\sum_k\dot{\phi}^{k}-1\right) -2\beta s\Phi^{-2}\sum_k\dot{\phi}^k\tau_k^2\nonumber\\
  &+C(\beta+1) s\Phi^{-2}\sum_k\dot{\phi}^{k}+C\beta s\Phi^{-1},
\end{align}
where we used the $C^0$ and $C^1$ estimates of $s$ in the last inequality, and the constants $C$ in the last line of \eqref{s7:evl-S} depend on $F$ and the initial hypersurface $M_0$. We next apply the maximum principle to \eqref{s7:evl-S} to obtain a uniform upper bound on $\tau_{11}$. Note that $\tau_{11}=\tau_1$ is the largest eigenvalue of $(\tau_{ij})$ at the point $(z_1,t_1)$.

We first estimate the gradient terms in \eqref{s7:evl-S}.  By \eqref{s5:phi-d2} and \eqref{s5:phi-p2}, we have
\begin{align}\label{s5:LTE1}
  \ddot{\Phi}^{k\ell,pq}\bar{\nabla}_1\tau_{k\ell}\bar{\nabla}_1\tau_{pq}= &~\ddot{\phi}^{k\ell}\bar{\nabla}_1\tau_{kk}\bar{\nabla}_1\tau_{\ell\ell}+2\sum_{k>\ell}\frac{\dot{\phi}^k-\dot{\phi}^\ell}{\tau_k-\tau_\ell}(\bar{\nabla}_1\tau_{k\ell})^2\nonumber\\
   \leq &~ 2\sum_{k>1}\frac{\dot{\phi}^k-\dot{\phi}^1}{\tau_k-\tau_1}(\bar{\nabla}_1\tau_{k1})^2\nonumber\\ \leq &~ -2\sum_{k>1}\frac 1{\tau_1}(\dot{\phi}^k-\dot{\phi}^1)(\bar{\nabla}_1\tau_{k1})^2.
\end{align}
The Codazzi equation \eqref{s6:Codaz} implies that
\begin{align}\label{s5:LTE2}
  (\bar{\nabla}_1\tau_{k1})^2= & \left( \bar{\nabla}_k\tau_{11}+\frac 12Q_{11p}\tau_{kp}-\frac 12Q_{k1p}\tau_{1p}\right)^2
  \geq   \frac 12(\bar{\nabla}_k\tau_{11})^2-C(\tau_{11})^2.
\end{align}
Substituting \eqref{s5:LTE2} into \eqref{s5:LTE1}, and noting that
\begin{equation*}
  \frac 1{\tau_{11}}\bar{\nabla}_k\tau_{11}=\alpha\bar{\nabla}_ks-\beta \bar{\nabla}_k|\bar{\nabla}s|^2, \quad k=1,\cdots,n
\end{equation*}
holds at $(z_1,t_1)$,   we have
\begin{align*}
   & \frac 1{\tau_{11}}\ddot{\Phi}^{k\ell,pq}\bar{\nabla}_1\tau_{k\ell}\bar{\nabla}_1\tau_{pq}+\frac 1{(\tau_{11})^2}\sum_k\dot{\phi}^k(\bar{\nabla}_k\tau_{11})^2\nonumber\\
   \leq &~-2\sum_{k>1}\frac 1{\tau_{11}^{2}}(\dot{\phi}^k-\dot{\phi}^1)\left(\frac 12(\bar{\nabla}_k\tau_{11})^2-C(\tau_{11})^2\right)+\frac 1{(\tau_{11})^2}\sum_k\dot{\phi}^k(\bar{\nabla}_k\tau_{11})^2\nonumber\\
= &~\dot{\phi}^1\tau_{11}^{-2}\sum_k(\bar{\nabla}_k\tau_{11})^2+2C\sum_{k>1}(\dot{\phi}^k-\dot{\phi}^1)\nonumber\\
   = &~\dot{\phi}^1\sum_k\left(\alpha\bar{\nabla}_ks-\beta \bar{\nabla}_k|\bar{\nabla}s|^2\right)^2+2C\sum_{k>1}(\dot{\phi}^k-\dot{\phi}^1)
\end{align*}
at $(z_1,t_1)$. By the definition of $\tau_{ij}$ together with the $C^0$ and $C^1$ estimates of $s$, we have
\begin{equation*}
  \left(\alpha\bar{\nabla}_ks-\beta \bar{\nabla}_k|\bar{\nabla}s|^2\right)^2\leq C\left(\alpha^2+\beta^2(\tau_1^2+1)\right),\quad k=1,\cdots,n
\end{equation*}
for some constant $C$ depending only on $M_0$ and $F$. In summary, the gradient terms on the right-hand side of \eqref{s7:evl-S} satisfies
\begin{align}\label{s6:gradt}
& s\Phi^{-2}\left(\frac 1{\tau_{11}}\ddot{\Phi}^{k\ell,pq}\bar{\nabla}_1\tau_{k\ell}\bar{\nabla}_1\tau_{pq}+\frac 1{(\tau_{11})^2}\dot{\Phi}^{k\ell}\bar{\nabla}_k\tau_{11}\bar{\nabla}_\ell\tau_{11}\right)\nonumber\\
 \leq  &  s\Phi^{-2}\left(C\dot{\phi}^1\left(\alpha^2+\beta^2(\tau_1^2+1)\right)+C\sum_k\dot{\phi}^k\right)
\end{align}
at $(z_1,t_1)$.

Substituting \eqref{s6:gradt} into \eqref{s7:evl-S}, and applying the maximum principle, we have
\begin{align*}
0\leq &~s\Phi^{-2}\left(C\dot{\phi}^1\left(\alpha^2+\beta^2(\tau_1^2+1)\right)+C\sum_k\dot{\phi}^k\right)\nonumber\\
 &-\alpha\left(1- s{\Phi}^{-1}\right)^2-\alpha s^2\Phi^{-2}\left(\sum_k\dot{\phi}^{k}-1\right) -2\beta s\Phi^{-2}\sum_k\dot{\phi}^k\tau_k^2\nonumber\\
  &+C(\beta+1) s\Phi^{-2}\sum_k\dot{\phi}^{k}+C\beta s\Phi^{-1}
\end{align*}
at $(z_1,t_1)$. Multiplying by $s^{-1}\Phi^2$ the above inequality, rearranging the terms and using $\Phi\leq \tau_{11}$, we have
\begin{align*}
   & 2\beta \sum_k\dot{\phi}^k\tau_k^2+\alpha s\left(\frac{\Phi}s-1\right)^2+\alpha s\left(\sum_k\dot{\phi}^{k}-1\right) \\
  \leq & C\dot{\phi}^1\left(\alpha^2+\beta^2(\tau_1^2+1)\right)+C(\beta+1)\sum_k\dot{\phi}^k+C\beta\Phi
\end{align*}
at $(z_1,t_1)$. Since $0<r\leq s(z,t)\leq R$ for any $(z,t)\in \Sigma_F\times [0,T)$, we have
\begin{align}\label{s6:tau-pf1}
   & 2\beta \sum_k\dot{\phi}^k\tau_k^2+\alpha r\left(\frac{\Phi}s-1\right)^2+\alpha r\left(\sum_k\dot{\phi}^{k}-1\right) \nonumber\\
  \leq & C\beta^2\dot{\phi}^1\tau_1^2+C\dot{\phi}^1\left(\alpha^2+\beta^2\right)+C(\beta+1)\sum_k\dot{\phi}^k+C\beta\Phi
\end{align}
at $(z_1,t_1)$.

Now, we apply an observation in \cite{Guanbo15} and \cite{Xia17} which relies on the property that $\phi$ is increasing and concave in $\Gamma_+$, and vanishes on $\partial\Gamma_+$. Precisely, for $\phi=\left({E_{n}}/{E_{n-k}}\right)^{1/k}$ we have the following.
\begin{prop}[Lemma 2.2 in \cite{Guanbo15}]\label{s7:prop-Guan}
Let $K\subset \Gamma_+$ be a compact set and $\gamma>0$. Then there exists a constant $\theta>0$ depending only on $K$ and $\gamma$ such that for $\tau\in \Gamma_+$ and $\mu\in K$ satisfying $|\nu_{\tau}-\nu_\mu|\geq \gamma$, we have
\begin{equation*}
  \sum_{k}\dot{\phi}^k(\tau)\mu_k-\phi(\mu)\geq \theta \left(\sum_k\dot{\phi}^k(\tau)+1\right),
\end{equation*}
where $\nu_\tau$ denotes the unit normal vector of the level set $\{x\in \Gamma_+:~\phi(x)=\phi(\tau)\}$ at the point $\tau$, i.e. $\nu_\tau=D\phi(\tau)/|D\phi(\tau)|$.
\end{prop}

Let $\mu=r(1,\cdots,1)\in \Gamma_+$, where $r$ is the lower bound of the anisotropic support function $s(z,t)$, and let $K=\{\mu\}$ in Proposition~\ref{s7:prop-Guan}. There exists a small constant $\gamma>0$ such that $\nu_\mu-2\gamma (1,\cdots,1)\in \Gamma_+$. We have two cases. Firstly, if the anisotropic principal radii of curvature $\tau$ at the point  $(z_1,t_1)$  satisfy $|\nu_{\tau}-\nu_\mu|\geq \gamma$, Proposition~\ref{s7:prop-Guan} implies that
\begin{equation}\label{s6:tau-pf2}
 \sum_k\dot{\phi}^k(\tau)-1\geq \frac {\theta} r \left(\sum_k\dot{\phi}^k(\tau)+1\right).
\end{equation}
Substituting \eqref{s6:tau-pf2} into \eqref{s6:tau-pf1}, we have
\begin{align*}
   & 2\beta \sum_k\dot{\phi}^k\tau_k^2+\alpha r\left(\frac{\Phi}s-1\right)^2+\alpha \theta\left(\sum_k\dot{\phi}^{k}+1\right) \\
  \leq & C\beta^2\dot{\phi}^1\tau_1^2+C\dot{\phi}^1\left(\alpha^2+\beta^2\right)+C(\beta+1)\sum_k\dot{\phi}^k+C\beta \Phi.
\end{align*}
We first choose $\beta$ small such that
\begin{equation*}
   C\beta^2\dot{\phi}^1\tau_1^2\leq \beta \sum_k\dot{\phi}^k\tau_k^2,
\end{equation*}
and then choose $\alpha$ large such that
\begin{align*}
 C(\beta+1)\sum_k\dot{\phi}^k+ C\beta\Phi \leq  & \alpha r\left(\frac{\Phi}s-1\right)^2+\alpha \theta\left(\sum_k\dot{\phi}^{k}+1\right).
\end{align*}
Note that the constants $\alpha$ and $\beta$ can  be chosen depending only on $M_0$ and $F$. Then we have
\begin{equation*}
\beta \dot{\phi}^1\tau_1^2\leq C\dot{\phi}^1\left(\alpha^2+\beta^2\right),
\end{equation*}
which gives an upper bound on $\tau_1$.

Secondly, if the anisotropic principal radii of curvature $\tau$ at the point  $(z_1,t_1)$  satisfy $|\nu_{\tau}-\nu_\mu|\leq \gamma$, it follows from $\nu_\mu-2\gamma (1,\cdots,1)\in \Gamma_+$ that $\nu_\tau-\gamma (1,\cdots,1)\in \Gamma_+$ and so $\dot{\phi}^i\geq C'(\sum_k(\dot{\phi}^k)^2)^{1/2}\geq C\sum_k\dot{\phi}^k$ at $\tau$ for some constants $C'$, $C$ and all $i=1,\cdots,n$. Thus
\begin{equation*}
  \Phi=\sum_k\dot{\phi}^k\tau_k\leq (\sum_k\dot{\phi}^k)\tau_1\leq C\dot{\phi}^1\tau_1.
\end{equation*}
Then we can discard the terms involving $\alpha$ on the left-hand side of \eqref{s6:tau-pf1} and obtain
\begin{align}\label{s6:tau-pf3}
 2\beta \dot{\phi}^1\tau_1^2\leq & C\beta^2\dot{\phi}^1\tau_1^2+C\dot{\phi}^1\left(\alpha^2+\beta^2\right)+C(\beta+1)\sum_k\dot{\phi}^k+C\beta \Phi\nonumber\\
 \leq &\dot{\phi}^1\left(C\beta^2\tau_1^2+C\left(\alpha^2+\beta^2\right)+C(\beta+1) +C\beta\tau_1\right).
\end{align}
Choosing $\beta$ small such that
\begin{equation*}
   C\beta^2\tau_1^2\leq \beta \tau_1^2
\end{equation*}
and multiplying by $1/{\dot{\phi}^1}$ both sides of \eqref{s6:tau-pf3}, we have
\begin{align}\label{s6:tau-pf4}
 \beta \tau_1^2\leq & C\left(\alpha^2+\beta^2\right)+C(\beta+1) +C\beta\tau_1.
\end{align}
This implies that $\tau_1$ is bounded from above.

Therefore, we can choose a small constant $\beta$ and a large constant $\alpha$ in the definition \eqref{s6:omeg-def} of $\omega(z,t)$ such that in both the above two cases, the largest anisotropic principal radius of curvature $\tau_1$ has a uniform upper bound at $(z_1,t_1)$. By the definition of $\omega(z,t)$ together with the $C^0$ and $C^1$ estimates, we conclude that $\tau_1$ is bounded from above uniformly along the flow \eqref{flow-1}. This completes the proof of Proposition \ref{s7:prop-tau}.
\end{proof}

Combining Proposition \ref{s7:lem-Psi} and Proposition \ref{s7:prop-tau}, we obtain the uniform curvature estimate of the solution $M_t$ to the flow \eqref{flow-1}. In fact, since
\begin{align*}
 E_k\geq &{\binom nk}^{-1}\kappa_n\cdots \kappa_{n-k+1}  \geq  {\binom nk}^{-1}\kappa_n \kappa_1^{k-1},
\end{align*}
the uniform upper bound on $E_k(\kappa)$ and the uniform positive lower bound on the anisotropic principal curvature $\kappa_i$  imply that
\begin{equation}\label{s6.curv-est}
  0<\frac 1C\leq \kappa_i\leq C,\qquad i=1,\cdots,n
\end{equation}
for some constant $C$ depending only on $F$ and $M_0$.

\section{Monotonicity of the isoperimetric ratio}\label{sec:monot}
In this section we show that the following higher order isoperimetric ratio
\begin{equation*}
  \mathcal{I}_k(\Omega,W_F)=\frac{V_{n+2-k}(\Omega, W_F)}{V_{n+1}(\Omega,W_F)^{\frac{n+2-k}{n+1}}},\quad k=2,\cdots,n
\end{equation*}
is monotone non-increasing along the flow \eqref{flow-1}. Note that the power to $V_{n+1}(\Omega,W_F)$ is chosen such that $\mathcal{I}_k(\Omega,W_F)$ is invariant under the scaling of the domain $\Omega$.
\begin{prop}
Let $\Omega_t$ be the domain enclosed by the solution $M_t$ of the flow \eqref{flow-1}. Then $ \mathcal{I}_k(\Omega_t,W_F)$ is monotone non-increasing in time $t$.
\end{prop}
\proof
By the Minkowski identity \eqref{eq2.Minkowski} for $k=0$, we have
\begin{equation*}
  \int_{M_t}d\mu_F=\int_{M_t}E_1(\kappa)\sigma_Fd\mu_F.
\end{equation*}
Then the equation \eqref{s3:Vol-1} implies that along the flow \eqref{flow-1},  the volume of $\Omega_t$ satisfies
\begin{align}\label{s7:mon-1}
  \frac{d}{dt} \mathrm{Vol}(\Omega_t)~=&~\int_{M_t}(1-E_k^{1/k}(\kappa)\sigma_F)d\mu_F\nonumber \\
= &~\int_{M_t}(E_1(\kappa)-E_k^{1/k}(\kappa))\sigma_Fd\mu_F.
\end{align}
Since the anisotropic support function $\sigma_F$ is positive along the flow by Lemma \ref{s4:lem-sigm}, the Newton--MacLaurin inequality (see Lemma 2.5 in \cite{Guan14})
\begin{equation}\label{s7:Newt}
  E_1(\kappa)~\geq ~E_2^{1/2}(\kappa)~\geq ~ E_k^{1/k}(\kappa),\quad k=2,\cdots,n
\end{equation}
implies that the right-hand side of \eqref{s7:mon-1} is positive unless $M_t$ is  anisotropically totally umbilical.

On the other hand, by \eqref{s3:evl-Vk}
\begin{align}\label{s7:mon-2}
  \frac{d}{dt} V_{n+2-k}(\Omega_t, W_F) =& (n+2-k)\int_{M_t}E_{k-1}(\kappa)\left(1-E_k^{1/k}(\kappa)\sigma_F\right)d\mu_F\nonumber\\
  \leq &(n+2-k)\int_{M_t}\left(E_{k-1}(\kappa)-E_k(\kappa)\sigma_F\right)d\mu_F\nonumber\\
  =& 0,
\end{align}
where we used \eqref{s7:Newt} and \eqref{eq2.Minkowski} again. Combining \eqref{s7:mon-1}, \eqref{s7:mon-2} and the fact $V_{n+1}(\Omega_t,W_F)=(n+1) \mathrm{Vol}(\Omega_t)$  implies that $ \mathcal{I}_k(\Omega_t,W_F)$ is monotone non-increasing along the flow \eqref{flow-1}.
\endproof

\section{Long time existence and smooth convergence}\label{sec:LTE}

In this section we prove that the solution $M_t$ of the flow \eqref{flow-1} exists for all positive time and converges smoothly to a scaled Wulff shape centered at the origin.
\subsection{Long time existence}\label{sec:8-1}
In \S \ref{sec:C2} we have obtained  the uniform estimate
\begin{equation*}
  0<\frac 1C\leq \kappa_i\leq C,\qquad i=1,\cdots,n
\end{equation*}
on the anisotropic principal curvature $\kappa_i$ of $M_t$ for some constant $C$ depending only on $F$ and $M_0$. This together with the $C^0, C^1$ estimates in \S \ref{sec:C01} yields the $C^2$ estimate of the solution $M_t$. Since $E_k^{1/k}$ is homogeneous of degree one and concave with respect to its argument, and the anisotropic support function $\sigma_F$ is uniformly bounded from below and above by positive constants, the standard Krylov--Safonov theory and Schauder theory (see \cite{Lie96}) imply that the solution $M_t$ has uniform $C^{k,\alpha}$ estimates for all $k\geq 2$ and some $\alpha\in (0,1)$. This guarantees that the solution $M_t$ exists for all time $t\in [0, \infty)$ and has uniform regularity estimates.
\subsection{Smooth convergence}
We need to show that the solution $M_t$ converges to a scaled Wulff shape centered at the origin as the time $t\to\infty$. Along the flow \eqref{flow-1}, we know from \S \ref{sec:monot} that the volume of the enclosed domain $\Omega_t$ is monotone non-decreasing. By the $C^0$ estimate, the volume of $\Omega_t$ is bounded from above by the volume of $RW_F$, where $R$ is the constant in \S \ref{sec:C0}. Integrating \eqref{s7:mon-1} over $[0,\infty)$, we get
\begin{equation}\label{s8.1}
0\leq \int_0^{\infty}\int_{M_t}\left(E_1-E_k^{1/k}\right)\sigma_Fd\mu_F dt~\leq ~R^{n+1}|W_F|-|\Omega_0|~<\infty.
\end{equation}
By the uniform regularity estimates of $M_t$, the integrand in the time integral of \eqref{s8.1} is uniformly continuous in $t$. This implies that
\begin{equation*}
  0\leq \int_{M_t}\left(E_1-E_k^{1/k}\right)\sigma_Fd\mu_F ~\to ~ 0,\qquad \mathrm{as} ~t\to\infty.
\end{equation*}
Since $k\geq 2$, the Newton--MacLaurin inequality \eqref{s7:Newt} implies that
\begin{align*}
  0\leq &\int_{M_t}\left(E_1-E_2^{1/2}\right)\sigma_Fd\mu_F \leq ~\int_{M_t}\left(E_1-E_k^{1/k}\right)\sigma_Fd\mu_F  ~\to ~ 0,\qquad \mathrm{as} ~t\to\infty.
\end{align*}
It follows from the uniform regularity estimates, the positive bound on $\sigma_F$ and the interpolation inequalities that
\begin{align*}
  0~\leq &~E_1-E_2^{1/2} ~\to ~ 0,\quad \mathrm{uniformly~in} ~C^\infty \quad \mathrm{as} ~t\to\infty.
\end{align*}
This is equivalent to that
\begin{align}\label{s8.3}
 \sum_{k<\ell}|\kappa_k-\kappa_\ell|^2= &n^2(n-1)(E_1^2-E_2)\nonumber\\
 =&n^2(n-1)(E_1+E_2^{1/2})(E_1-E_2^{1/2})~\to ~ 0
\end{align}
uniformly in $C^\infty$ as the time $t\to\infty$, where we used the uniform positive bounds on $E_1+E_2^{1/2}$.

The uniform regularity estimates of $M_t$ imply that for any sequence of times $\{t_k\}$ tending to the infinity, there exists a subsequence of times $\{t_{k_j}\}$ such that $M_{t_{k_j}}$ converges to a limit hypersurface $M_\infty$ smoothly. Using the anisotropic Gauss parametrization
\begin{equation}\label{s8.2-s}
  \frac{\partial}{\partial t}s=1-\frac s{\Phi(\tau_{ij})},
\end{equation}
of the flow \eqref{flow-1}, we see that the spatial maximum $s_{\max}$ of the anisotropic support function $s(z,t)$ of $M_t$ is non-increasing in time, while the spatial minimum $s_{\min}$ is non-decreasing in time. Then the sub-convergence of $M_{t_k}$ to $M_\infty$ implies that $M_t$ converges to $M_{\infty}$ in $C^0$ for all time $t\to\infty$. The smooth convergence of $M_t$ to $M_\infty$ then follows from the interpolation inequality and the uniform $C^k$ estimate for all $k\geq 0$.  By \eqref{s8.3} the limit hypersurface $M_\infty$ is anisotropically totally umbilical and therefore is a scaled Wulff shape of some radius $\bar{r}$.   The limit Wulff shape must be centered at the origin, since otherwise $s_{\max}$ is strictly decreasing and $s_{\min}$ is strictly increasing, and then the flow will move the center of the Wulff shape towards the origin. Therefore, we conclude that the solution $M_t$ converges smoothly as $t\to\infty$ to the Wulff shape $\bar{r}\Sigma_F$ centered at the origin.

\subsection{Proof of Corollary \ref{s1-cor1}}
The smooth convergence of the flow \eqref{flow-1} together with the monotonicity of the isoperimetric ratio $ \mathcal{I}_k(\Omega_t,W_F)$ can be used to derive the conclusion in Corollary \ref{s1-cor1}.
In fact, for any smooth, closed and strictly convex hypersurface $M_0=\partial\Omega_0$ in $\mathbb{R}^{n+1}$, we run the flow \eqref{flow-1} starting from $M_0$. Theorem \ref{s1-thm1} implies that the solution $M_t$ converges to a scaled Wulff shape as $t\to\infty$. By the monotonicity of $ \mathcal{I}_k(\Omega_t,W_F)$, we have
\begin{align*}
  \mathcal{I}_k(\Omega_0,W_F)\geq &~\lim_{t\to\infty} \mathcal{I}_k(\Omega_t,W_F)=  \mathcal{I}_k(\bar{r}W_F,W_F)=\left((n+1)|W_F|\right)^{\frac{k-1}{n+1}}.
\end{align*}
Then the definition \eqref{s1.Ik} of $ \mathcal{I}_k(\Omega_0,W_F) $ implies that
\begin{equation}\label{s8.4}
  \int_{M_0}E_{k-2}(\kappa)d\mu_F\geq (n+1)|\Omega_0|^{\frac{n+2-k}{n+1}}|W_F|^{\frac{k-1}{n+1}},\quad k=2,\cdots,n,
\end{equation}
which is equivalent to \eqref{s1:AF} in Corollary \ref{s1-cor1}. If the equality holds in \eqref{s8.4}, the proof of the monotonicity of $ \mathcal{I}_k(\Omega_t,W_F)$ given in \S \ref{sec:monot} implies that each $M_t$ is anisotropically totally umbilical. In particular, $M_0$ is a scaled Wulff shape.

\section{Exponential convergence}\label{sec:exp}
In this section, we prove the exponential convergence of the flow \eqref{flow-1} and complete the proof of Theorem \ref{s1-thm1}.  This will be obtained by studying the linearization of \eqref{s8.2-s}.
\begin{lem}
The linearization of \eqref{s8.2-s} around $s\equiv \bar{r}$ is
\begin{equation}\label{s9.liner}
   \frac{\partial}{\partial t}h(z,t)=\frac{1}{n\bar{r}}\left(\bar{\Delta} h(z,t)-\frac{1}{2}\bar{g}^{ij}Q_{ijk}\bar{\nabla}_kh(z,t)\right).
\end{equation}
The right-hand side is a self-adjoint non-positive operator with respect to the inner product $\langle \phi,\varphi\rangle=\int_{\Sigma_F}\phi\varphi d\mu_F$.
\end{lem}
\proof Write $s(z,t)=\bar{r}+\varepsilon h(z,t)$ for small $\varepsilon$. We have
\begin{align*}
  \frac{d}{d\varepsilon}\bigg|_{\varepsilon=0}s(z,t) =& h(z,t), \\
   \frac{d}{d\varepsilon}\bigg|_{\varepsilon=0} \Phi(\tau_{ij}[z,t])=& \dot{\Phi}^{ij}(\bar{r}\bar{g})\tau_{ij}[h(z,t)]\\
   =&\frac{1}{n}\left(\bar{\Delta} h(z,t)+nh(z,t)-\frac{1}{2}\bar{g}^{ij}Q_{ijk}\bar{\nabla}_kh(z,t)\right).
\end{align*}
Then the linearization of \eqref{s8.2-s} is given by
 \begin{align*}
  \frac{\partial}{\partial t}h(z,t)= & -\frac{1}{\bar{r}}h(z,t)+\frac{1}{\bar{r}n}\left(\bar{\Delta} h(z,t)+nh(z,t)-\frac{1}{2}\bar{g}^{ij}Q_{ijk}\bar{\nabla}_kh(z,t)\right) \\
   = & \frac{1}{n\bar{r}}\left(\bar{\Delta} h(z,t)-\frac{1}{2}\bar{g}^{ij}Q_{ijk}\bar{\nabla}_kh(z,t)\right).
 \end{align*}
The operator on the right-hand side is self-adjoint and non-positive; see \cite[Lemma 2.8]{Xia13}.
\endproof

Let
\begin{equation}\label{s9.L}
  \mathcal{L}h(z,t)=\bar{\Delta} h(z,t)-\frac{1}{2}\bar{g}^{ij}Q_{ijk}\bar{\nabla}_kh(z,t),
\end{equation}
which is self-adjoint. Then there exists a sequence of eigenvalues $0=\lambda_0<\lambda_1\leq \cdots \leq \lambda_k\to\infty$ of the operator $\mathcal{L}$ with corresponding eigenfunctions $\phi_k$ satisfying
\begin{equation*}
  \mathcal{L}\phi_k=-\lambda_k\phi_k.
\end{equation*}
We may choose $\{\phi_k\}_{k\geq 0}$ as an orthonormal basis of $L^2(\Sigma_F)$ with respect to the inner product $\langle \phi,\varphi\rangle=\int_{\Sigma_F}\phi\varphi d\mu_F$.

Since $s(z,t)$ converges to $\bar{r}$ smoothly as $t\to\infty$, for sufficiently large time $t$ we know that $\sigma(z,t):=s(z,t)-\bar{r}$ is sufficiently small. We would like to show that $\sigma(z,t)$ converges to zero exponentially. Applying the linearization \eqref{s9.liner}, we have
\begin{align*}
  \frac{\partial}{\partial t}\sigma(z,t) =& \frac{1}{n\bar{r}}\mathcal{L} \sigma(z,t) +O(|\sigma(z,t)|_{C^2}^2).
\end{align*}
We decompose $\sigma(z,t)$ in the following form
\begin{equation*}
  \sigma(z,t)=\sum_{k=0}^\infty \varphi_k(z,t)
\end{equation*}
with $\varphi_k(z,t)=f_k(t)\phi_k(z)$ being an eigenfunction of $\mathcal{L}$ corresponding to the eigenvalue $\lambda_k$, where $f_k(t)$ is just the coefficient of the decomposition of $\sigma(z,t)$ with respect to the basis $\{\phi_k\}_{k\geq 0}$ for each fixed time $t$. Note that $\varphi_0(z,t)=f_0(t)\phi_0$ is a constant function on $\Sigma_F$ for each $t$ and satisfies
\begin{align*}
  \varphi_0=&\frac{1}{|\Sigma_F|_F }\int_{\Sigma_F}\sigma d\mu_F.
\end{align*}
Then
\begin{align}\label{s8.3-1}
  \frac{1}{2}\frac{d}{d t}\int_{\Sigma_F}\sigma(z,t)^2d\mu_F \leq & \frac{1}{n\bar{r}}\int_{\Sigma_F}\sigma(z,t)\mathcal{L} \sigma(z,t)d\mu_F +C\int_{\Sigma_F} |\sigma||\sigma(z,t)|_{C^2}^2d\mu_F\nonumber\\
  =& \frac{1}{n\bar{r}}\int_{\Sigma_F}-\sum_{k\geq 1}\lambda_k\varphi_k^2(z,t)d\mu_F+C\int_{\Sigma_F} |\sigma||\sigma(z,t)|_{C^2}^2d\mu_F\nonumber\\
  =&- \frac{1}{n\bar{r}}\int_{\Sigma_F}\lambda_1\sum_{k\geq 0}\varphi_k^2(z,t)d\mu_F+\frac{1}{n\bar{r}}\int_{\Sigma_F}\lambda_1\varphi_0^2(z,t)d\mu_F\nonumber\\
  &\quad -\frac{1}{n\bar{r}}\int_{\Sigma_F}\sum_{k\geq 2}(\lambda_k-\lambda_1)\varphi_k^2(z,t)d\mu_F+C\int_{\Sigma_F} |\sigma||\sigma(z,t)|_{C^2}^2d\mu_F\nonumber\\
  \leq & - \frac{\lambda_1}{n\bar{r}}\int_{\Sigma_F}\sigma(z,t)^2d\mu_F+\underbrace{\frac{\lambda_1}{n\bar{r}|\Sigma_F|_F}\left(\int_{\Sigma_F}\sigma(z,t)d\mu_F\right)^2}_I\nonumber\\
  &\quad +C\underbrace{\int_{\Sigma_F} |\sigma||\sigma(z,t)|_{C^2}^2d\mu_F}_{II}.
\end{align}
We will show that the terms $I, II$ decay faster than  $\int_{\Sigma_F}\sigma(z,t)^2d\mu_F$. Then the above inequality \eqref{s8.3-1} implies the exponential decay of $\int_{\Sigma_F}\sigma(z,t)^2d\mu_F$.

To estimate the term $I$, we explore the monotonicity of the volume $V_{n+1}(\Omega,W_F)$ and the mixed volume $V_{n+2-k}(\Omega,W_F)$. Under the anisotropic Gauss map parametrization,
\begin{align*}
  V_{n+1}(\Omega,W_F)= & \int_{\Sigma_F}sE_n(\tau[s])d\mu_F, \\
   V_{n+2-k}(\Omega,W_F)=& \int_{\Sigma_F}sE_{n-k+1}(\tau[s])d\mu_F,
\end{align*}
where $s=s(z)$ is the anisotropic support function of $\partial\Omega$ and $d\mu_F$ is the anisotropic area form on the Wulff shape $\Sigma_F$. By the monotonicity, we have
\begin{align}
   V_{n+1}(\Omega_t,W_F)\leq & V_{n+1}(\bar{r}W_F,W_F)=\bar{r}^{n+1}|\Sigma_F|_F,\label{s8.3-V}\\ V_{n+2-k}(\Omega_t,W_F)\geq & V_{n+2-k}(\bar{r}W_F,W_F)=\bar{r}^{n+2-k}|\Sigma_F|_F.\label{s8.3-Vk}
\end{align}
Expanding the left-hand side of \eqref{s8.3-V} and \eqref{s8.3-Vk} using the fact $\sigma(z,t)=s(z,t)-\bar{r}$ is sufficiently small, we obtain
\begin{align}\label{s8.3-Vb}
  0~\geq  & ~V_{n+1}(\Omega_t,W_F)- \bar{r}^{n+1}|\Sigma_F|_F\nonumber\\
  = & \int_{\Sigma_F}sE_n(\tau[s])d\mu_F- \bar{r}^{n+1}|\Sigma_F|_F\nonumber\\
  =& \int_{\Sigma_F} \left(\bar{r}+\sigma +O(\sigma^2)\right)\left(\bar{r}^n+ \bar{r}^{n-1}(\bar{L}\sigma+n\sigma)+\bar{r}^{n-2}O(|\sigma|_{C^2}^2)\right)d\mu_F\nonumber\\
  &\quad - \bar{r}^{n+1}|\Sigma_F|_F\nonumber\\
  =&(n+1)\bar{r}^n \int_{\Sigma_F} \sigma d\mu_F+\int_{\Sigma_F} O(|\sigma|_{C^2}^2) d\mu_F
\end{align}
and
\begin{align}\label{s8.3-Vk-b}
  0~\leq  & ~V_{n+2-k}(\Omega_t,W_F)- \bar{r}^{n+2-k}|\Sigma_F|_F\nonumber\\
  = & \int_{\Sigma_F}sE_{n-k+1}(\tau[s])d\mu_F- \bar{r}^{n+2-k}|\Sigma_F|_F\nonumber\\
  =& \int_{\Sigma_F} \left(\bar{r}+\sigma +O(\sigma^2)\right)\left(\bar{r}^{n-k+1}+\frac{n-k+1}{n}\bar{r}^{n-k}(\bar{L}\sigma+n\sigma)+\bar{r}^{n-k-1}O(|\sigma|_{C^2}^2)\right)d\mu_F\nonumber\\
  &\quad - \bar{r}^{n+2-k}|\Sigma_F|_F\nonumber\\
  =&(n-k+2)\bar{r}^{n-k+1} \int_{\Sigma_F} \sigma d\mu_F+\int_{\Sigma_F} O(|\sigma|_{C^2}^2) d\mu_F.
\end{align}
Combining \eqref{s8.3-Vb} and \eqref{s8.3-Vk-b} implies that
\begin{equation}\label{s8.3-I-est}
  \left(\int_{\Sigma_F} \sigma d\mu_F\right)^2\leq C\left(\int_{\Sigma_F}|\sigma|_{C^2}^2d\mu_F\right)^2.
\end{equation}
Then $I+II$ on the right-hand side of \eqref{s8.3-1} satisfies
\begin{equation}\label{s8.3-I-II}
  I+II\leq ~C\left(\int_{\Sigma_F}|\sigma|_{C^2}^2d\mu_F\right)^2+C\left(\int_{\Sigma_F}|\sigma|_{C^2}^4d\mu_F\right)^{1/2}\|\sigma\|_{L^2(\Sigma_F)}
\end{equation}
for some constant $C$.

We recall the following special case of Gagliardo--Nirenberg interpolation inequality; see \cite[Theorem 3.70]{Au98}.
\begin{lem}
For integers $j,k $ satisfying ${j}/{k}\leq \theta<1$, there exists a constant $C$ depending only on $j,k,\theta,n$ and the manifold $\Sigma$ such that
\begin{equation}\label{s8.3-GN}
  \|\bar{\nabla}^jf\|_{L^p(\Sigma)}\leq C \|\bar{\nabla}^kf\|_{L^2(\Sigma)}^{\theta}\|f\|_{L^2(\Sigma)}^{1-\theta}
\end{equation}
for any function $f\in W^{k,2}(\Sigma)$ with $\int_\Sigma fd\mu_F=0$, where
\begin{equation*}
  \frac{1}{p}=\frac{j}{n}+\frac{1}{2}-\frac{k}{n}\theta
\end{equation*}
for which $p$ is nonnegative.
\end{lem}
Though $\int_{\Sigma_F}\sigma d\mu_F\neq 0$, by setting
\begin{equation*}
  f=\sigma-\frac{1}{|\Sigma_F|_F}\int_{\Sigma_F}\sigma d\mu_F
\end{equation*}
in \eqref{s8.3-GN}, we see that \eqref{s8.3-GN} still holds for $\sigma$.   Using the uniform bound $\|\sigma\|_{C^k(\Sigma_F)}\leq C(k)$ for all $k\geq 0$, we obtain that
\begin{align}\label{s8.3-2a}
  \left(\int_{\Sigma_F}|\sigma|_{C^2}^2d\mu_F\right)^2= & \left(\int_{\Sigma_F}\left(|\bar{\nabla}^2\sigma|^2+|\bar{\nabla}\sigma|^2+|\sigma|^2\right)d\mu_F\right)^2 \nonumber\\
  \leq  & C\left(\|\sigma\|_{L^2(\Sigma_F)}^{2(1-\frac{2}{k})}+\|\sigma\|_{L^2(\Sigma_F)}^{2(1-\frac{1}{k})}+\|\sigma\|_{L^2(\Sigma_F)}^2\right)^2\nonumber\\
  \leq & C \|\sigma\|_{L^2(\Sigma_F)}^{4(1-\frac{2}{k})}\nonumber\\
 \leq & C\|\sigma\|_{L^2(\Sigma_F)}^3
\end{align}
by choosing $k\geq 8$. Similarly,  we have
\begin{align}\label{s8.3-2b}
  \left(\int_{\Sigma_F}|\sigma|_{C^2}^4d\mu_F\right)^{1/2}\leq  & C\left(\int_{\Sigma_F}\left(|\bar{\nabla}^2\sigma|^4+|\bar{\nabla}\sigma|^4+|\sigma|^4\right)d\mu_F\right)^{1/2} \nonumber \\
  \leq& C\left(\|\sigma\|_{L^2(\Sigma_F)}^{4(1-\frac{2+n/4}{k})}+\|\sigma\|_{L^2(\Sigma_F)}^{4(1-\frac{1+n/4}{k})}+\|\sigma\|_{C^0(\Sigma_F)}^2\|\sigma\|_{L^2(\Sigma_F)}^2\right)^{1/2}\nonumber\\
  \leq &C \|\sigma\|_{C^0(\Sigma_F)}^{1/2}\|\sigma\|_{L^2(\Sigma_F)}
\end{align}
by choosing $k\geq 8+n$. Substituting \eqref{s8.3-I-II} and \eqref{s8.3-2a}, \eqref{s8.3-2b} into \eqref{s8.3-1}, we get
\begin{align}\label{s8.3-e0}
  \frac{d}{d t}\|\sigma\|_{L^2(\Sigma_F)}\leq & - \frac{\lambda_1}{n\bar{r}}\|\sigma\|_{L^2(\Sigma_F)}+C\|\sigma\|_{L^2(\Sigma_F)}^{2}+C \|\sigma\|_{C^0(\Sigma_F)}^{1/2}\|\sigma\|_{L^2(\Sigma_F)}\nonumber\\
  =& - \frac{\lambda_1}{n\bar{r}}\|\sigma\|_{L^2(\Sigma_F)}\left(1-C\|\sigma\|_{L^2(\Sigma_F)}-C \|\sigma\|_{C^0(\Sigma_F)}^{1/2}\right).
\end{align}
For sufficiently large time $t$, we have
\begin{equation*}
  1-C\|\sigma\|_{L^2(\Sigma_F)}-C \|\sigma\|_{C^0(\Sigma_F)}^{1/2}\geq 1/2.
\end{equation*}
This implies that
\begin{equation}\label{s8.3-e1}
  \|\sigma\|_{L^2(\Sigma_F)}\leq C e^{-\frac{\lambda_1}{2n\bar{r}}t}.
\end{equation}
By Sobolev Embedding Theorem,
\begin{equation*}
  \|\sigma\|_{C^0(\Sigma_F)} \leq C \|\sigma\|_{W^{j,2}(\Sigma_F)}
\end{equation*}
for $j>n/2$. Interpolation inequality \eqref{s8.3-GN} together with the exponential decay \eqref{s8.3-e1} implies that
\begin{equation}\label{s8.3-e2}
   \|\sigma\|_{C^0(\Sigma_F)} \leq C e^{-\frac{\lambda_1}{2n\bar{r}}(1-\delta)t}
\end{equation}
for any $0<\delta \ll 1$. Substituting \eqref{s8.3-e1} and \eqref{s8.3-e2} into \eqref{s8.3-e0}, we can improve the decay rate in \eqref{s8.3-e1} to any constant smaller than $\frac{\lambda_1}{n\bar{r}}$. Applying the interpolation inequality \eqref{s8.3-GN} and Sobolev Embedding Theorem we eventually obtain that
 \begin{equation}\label{s8.3-e3}
   \|\sigma\|_{C^k(\Sigma_F)} \leq C(k,\delta) e^{-\frac{\lambda_1}{n\bar{r}}(1-\delta)t}
\end{equation}
for any given $0<\delta \ll 1$, where $\lambda_1>0$ is the first non-zero eigenvalue of the operator $\mathcal{L}$. This completes the proof of Theorem \ref{s1-thm1}.

\appendix
\section{Codazzi and Simons type equations}\label{sec:ap1}
In this appendix, we provide a detailed calculation to derive the Codazzi and Simons type equations in Lemma \ref{s6:lem1}.
\begin{lem}
We have the following Codazzi and Simons type equations for $\tau_{ij}$:
\begin{equation}\label{a1:Codaz}
  \bar{\nabla}_j\tau_{k\ell}+\frac 12Q_{k\ell p}\tau_{jp}~=~\bar{\nabla}_k\tau_{j\ell}+\frac 12Q_{j\ell p}\tau_{kp}
\end{equation}
and
\begin{align}\label{a1:simon}
  \bar{\nabla}_i \bar{\nabla}_j\tau_{k\ell} =&\bar{\nabla}_k\bar{\nabla}_\ell\tau_{ij}+\frac 12Q_{ijp}\bar{\nabla}_p\tau_{k\ell}-\frac 12Q_{k\ell p}\bar{\nabla}_p\tau_{ij}+\bar{g}_{ij}\tau_{k\ell}-\bar{g}_{kj}\tau_{i\ell}\nonumber\\
 &+\bar{g}_{i\ell}\tau_{kj}-\bar{g}_{k\ell}\tau_{ij} +\frac 14\left(Q_{jkq}Q_{ipq}-Q_{kpq}Q_{qij}\right)\tau_{p\ell}\nonumber\\
&+\frac 14\left(Q_{j\ell q}Q_{ipq}-Q_{\ell pq}Q_{qij}\right)\tau_{kp}+\frac 14\left(Q_{k \ell q}Q_{ipq}-Q_{kpq}Q_{qi \ell }\right)\tau_{pj}\nonumber\\
&+\frac 14\left(Q_{k \ell q}Q_{jpq}-Q_{kpq}Q_{qj \ell }\right)\tau_{ip}+\frac 12\bar{\nabla}_pQ_{ijk}\tau_{ \ell p}+\frac 12\bar{\nabla}_pQ_{ij \ell }\tau_{kp}\nonumber\\
&-\frac 12\bar{\nabla}_pQ_{jk \ell }\tau_{ip}-\frac 12\bar{\nabla}_pQ_{ik \ell }\tau_{jp}.
\end{align}
\end{lem}
\begin{proof}
(i) Taking the covariant derivative of the equation \eqref{s6:tau-def}, we have
\begin{align}
  \bar{\nabla}_j\tau_{k \ell }=& \bar{\nabla}_j\bar{\nabla}_k\bar{\nabla}_\ell s +\bar{g}_{k \ell }\bar{\nabla}_js-\frac 12\bar{\nabla}_jQ_{k \ell p}\bar{\nabla}_ps-\frac 12Q_{k \ell p}\bar{\nabla}_j\bar{\nabla}_ps,\label{a1:pf-1}\\
  \bar{\nabla}_k\tau_{j \ell }=& \bar{\nabla}_k\bar{\nabla}_j\bar{\nabla}_\ell s +\bar{g}_{j \ell }\bar{\nabla}_ks-\frac 12\bar{\nabla}_kQ_{j \ell p}\bar{\nabla}_ps-\frac 12Q_{j \ell p}\bar{\nabla}_k\bar{\nabla}_ps.
\end{align}
Using the Ricci identity and \eqref{s2:gauss-2}, we have
\begin{align}
  &\bar{\nabla}_j\bar{\nabla}_k\bar{\nabla}_\ell s-\bar{\nabla}_k\bar{\nabla}_j\bar{\nabla}_\ell s=\bar{R}_{jk \ell p}\bar{\nabla}_ps\nonumber\\
   =&\bar{g}_{j \ell }\bar{\nabla}_ks-\bar{g}_{k \ell }\bar{\nabla}_js+\frac 14Q_{k \ell q}Q_{qjp}\bar{\nabla}_ps-\frac 14Q_{j \ell q}Q_{qkp}\bar{\nabla}_ps.\label{a1:pf-2}
\end{align}
Combining \eqref{a1:pf-1}--\eqref{a1:pf-2} and using the total symmetry of $Q_{ijk}$ and $\bar{\nabla}_iQ_{jk \ell }$, we obtain
\begin{align*}
  \bar{\nabla}_j\tau_{k \ell }-\bar{\nabla}_k\tau_{j \ell }=& \frac 12Q_{j \ell p}\bar{\nabla}_k\bar{\nabla}_ps-\frac 12Q_{k \ell p}\bar{\nabla}_j\bar{\nabla}_ps +\frac 14Q_{k \ell q}Q_{qjp}\bar{\nabla}_ps-\frac 14Q_{j \ell q}Q_{qkp}\bar{\nabla}_ps\\
  =&\frac 12Q_{j \ell p}\bar{\nabla}_k\bar{\nabla}_ps-\frac 12Q_{k \ell p}\bar{\nabla}_j\bar{\nabla}_ps +\frac 12Q_{k \ell q}\left(\bar{\nabla}_q\bar{\nabla}_js+s\bar{g}_{qj}-\tau_{qj}\right)\\
   &\quad -\frac 12Q_{j \ell q}\left(\bar{\nabla}_q\bar{\nabla}_ks+s\bar{g}_{qk}-\tau_{qk}\right)\\
   =&\frac 12Q_{j \ell p}\tau_{pk}-\frac 12Q_{k \ell p}\tau_{pj}.
\end{align*}

(ii) For the Simons type equation, we take the covariant derivative of the equation \eqref{a1:Codaz}:
\begin{align*}
 \bar{\nabla}_i \bar{\nabla}_j\tau_{k \ell }=&\bar{\nabla}_i\bar{\nabla}_k\tau_{j \ell }+\frac 12\bar{\nabla}_i\left(Q_{j \ell q}\tau_{kq}-Q_{k \ell p}\tau_{jp}\right)\\
 =&\bar{\nabla}_k\bar{\nabla}_i\tau_{j \ell }+\bar{R}_{ikjp}\tau_{p \ell }+\bar{R}_{ik \ell p}\tau_{pj}+\frac 12\bar{\nabla}_i\left(Q_{j \ell q}\tau_{kq}-Q_{k \ell p}\tau_{jp}\right)\displaybreak[0]\\
 =&\bar{\nabla}_k\bar{\nabla}_\ell \tau_{ji}+\frac 12\bar{\nabla}_k\left(Q_{jiq}\tau_{ \ell q}-Q_{j \ell p}\tau_{ip}\right)+\bar{g}_{ij}\tau_{k \ell }-\bar{g}_{kj}\tau_{i \ell }+\bar{g}_{i \ell }\tau_{kj}-\bar{g}_{k \ell }\tau_{ij}\\
& +\frac 14\left(Q_{kjq}Q_{qip}-Q_{kpq}Q_{qij}\right)\tau_{p \ell }+\frac 14\left(Q_{k \ell q}Q_{qip}-Q_{kpq}Q_{qi \ell }\right)\tau_{pj}\\
&+\frac 12\bar{\nabla}_i\left(Q_{j \ell q}\tau_{kq}-Q_{k \ell p}\tau_{jp}\right)\displaybreak[0]\\
 =&\bar{\nabla}_k\bar{\nabla}_\ell \tau_{ji}+\bar{g}_{ij}\tau_{k \ell }-\bar{g}_{kj}\tau_{i \ell }+\bar{g}_{i \ell }\tau_{kj}-\bar{g}_{k \ell }\tau_{ij} +\frac 14\left(Q_{kjq}Q_{qip}-Q_{kpq}Q_{qij}\right)\tau_{p \ell }\\
&+\frac 14\left(Q_{k \ell q}Q_{qip}-Q_{kpq}Q_{qi \ell }\right)\tau_{pj}+\frac 12\bar{\nabla}_kQ_{jiq}\tau_{ \ell q}-\frac 12\bar{\nabla}_kQ_{j \ell p}\tau_{ip}\displaybreak[0]\\
&+\frac 12\bar{\nabla}_iQ_{j \ell q}\tau_{kq}-\frac 12\bar{\nabla}_iQ_{k \ell p}\tau_{jp}+\frac 12Q_{j \ell p}\left(\bar{\nabla}_i\tau_{kp}-\bar{\nabla}_k\tau_{ip}\right)\\
&+\frac 12Q_{jiq}\bar{\nabla}_k\tau_{ \ell q}-\frac 12Q_{k \ell p}\bar{\nabla}_i\tau_{jp}\displaybreak[0]\\
 =&\bar{\nabla}_k\bar{\nabla}_\ell \tau_{ji}+\bar{g}_{ij}\tau_{k \ell }-\bar{g}_{kj}\tau_{i \ell }+\bar{g}_{i \ell }\tau_{kj}-\bar{g}_{k \ell }\tau_{ij} +\frac 14\left(Q_{kjq}Q_{qip}-Q_{kpq}Q_{qij}\right)\tau_{p \ell }\\
&+\frac 14\left(Q_{k \ell q}Q_{qip}-Q_{kpq}Q_{qi \ell }\right)\tau_{pj}+\frac 12\bar{\nabla}_kQ_{jiq}\tau_{ \ell q}-\frac 12\bar{\nabla}_kQ_{j \ell p}\tau_{ip}\displaybreak[0]\\
&+\frac 12\bar{\nabla}_iQ_{j \ell q}\tau_{kq}-\frac 12\bar{\nabla}_iQ_{k \ell p}\tau_{jp}+\frac 14Q_{j \ell p}\left(Q_{ipq}\tau_{kq}-Q_{kpq}\tau_{iq}\right)\\
&+\frac 12Q_{jiq}\bar{\nabla}_q\tau_{k \ell }-\frac 12Q_{k \ell p}\bar{\nabla}_p\tau_{ij}+\frac 14Q_{jiq}\left(Q_{k \ell p}\tau_{pq}-Q_{ \ell qp}\tau_{kp}\right)\displaybreak[0]\\
&\quad -\frac 14Q_{k \ell p}\left(Q_{ijq}\tau_{pq}-Q_{jpq}\tau_{iq}\right).
\end{align*}
Using the total symmetry of $Q_{ijk}$ and $\bar{\nabla}_iQ_{jk \ell }$ and rearranging the terms gives the formula~\eqref{a1:simon}.
\end{proof}

\bibliographystyle{Plain}

\end{document}